\documentclass[fontsize=14pt,DIV=1]{scrartcl}

\usepackage{amsmath,amsfonts,amstext,amssymb,amsthm,flexisym,verbatim}
\usepackage{bbm,graphicx,empheq,marvosym,textcomp,stmaryrd}
\usepackage[cal=boondoxo]{mathalfa}
\usepackage[vcentermath]{genyoungtabtikz}
\usepackage{hyperref}
\usepackage{tikz}

\newtheorem*{thm*}{Theorem}
\newtheorem*{cor*}{Corollary}
\newtheorem*{ex*}{Example}
\newtheorem*{prop*}{Proposition}
\newtheorem*{lemma*}{Lemma}
\newtheorem*{rmk*}{Remark}

\newtheorem{cl}{Claim}[subsubsection]

\newtheoremstyle{named}{}{}{\itshape}{}{\bfseries}{.}{.5em}{#1 #3}
\theoremstyle{named}

\newtheorem*{namedprop}{Proposition}
\newtheorem*{namedcor}{Corollary}

\let\OLDthebibliography\thebibliography
\renewcommand\thebibliography[1]{
  \OLDthebibliography{#1}
  \small
  \setlength{\parskip}{3pt}
  \setlength{\itemsep}{3pt plus 0.3ex}
}

\def\sc{\mathrm{s.c.}}
\def\LSVK{\mathrm{LSVK}}

\def\tr{\operatorname{tr}}
\def\ct{\operatorname{ct}}
\def\Tab{\operatorname{Tab}}
\def\Mar{\mathfrak{M}\,}
\def\supp{\operatorname{supp}}
\def\Resc{\mathcal{R}}
\def\tab{\mathbf{T}}

\begin{document}
\title{Fluctuations of interlacing sequences}
\subtitle{
{{\normalsize\protect\textit{to Uzy Smilansky on his $75\frac23$-th birthday}}}
}
\setkomafont{subtitle}{\normalsize}
\author{\normalsize Sasha Sodin\textsuperscript{1}}
\date{\normalsize\today}
\maketitle
\footnotetext[1] {\,\,\,School of Mathematical
Sciences, Tel Aviv University, Tel Aviv, 69978, Israel $\binampersand$ School of Mathematical
Sciences, Queen Mary University of London, London E1~4NS, United Kingdom. E-mail:
sashas{\textoneoldstyle}{\MVAt}post.tau.ac.il. Supported in part by the European
Research Council start-up grant 639305 (SPECTRUM).}

\begin{abstract}
\vspace{-1.3cm}
In a series of works published in the 1990-s, Kerov put forth various applications
of the circle of ideas centred at the Markov moment problem to the limiting shape 
of random continual diagrams
arising in representation theory and spectral theory.
We demonstrate on several examples that his approach is also adequate to study the fluctuations about the limiting shape.

In the random matrix setting, we compare  two continual diagrams:
one is constructed from the
eigenvalues of the matrix and the critical points of its characteristic polynomial, whereas the second one is constructed from the eigenvalues of the matrix and those of its
principal submatrix.
The fluctuations of the latter diagram were recently studied by Erd\H{o}s and Schr\"oder;
we discuss the fluctuations of  the former, and compare the two limiting processes.

For Plancherel random partitions, the Markov correspondence establishes the
equivalence between Kerov's central limit theorem for the Young diagram  and the Ivanov--Olshanski central limit theorem for the transition measure. We outline a combinatorial proof of the latter, and compare the limiting process with the ones
of random matrices.
\end{abstract}

\section{Overview}

\subsection{Markov correspondence}\label{s:mc}
Two sequences of real numbers $A = (a_1 < \cdots < a_{n})$ and
$B = (b_1 < \cdots < b_{n-1})$ are called interlacing if
\[ a_1 < b_1 < a_2 < \cdots < b_{n-1} < a_{n}~. \]
To a pair of interlacing sequences $(A,B)$ one associates the probability measure
\[ \mu = \sum_{j=1}^{n} p_j \delta_{a_j}~, \]
where $p_j$ are defined by the simple fraction decomposition
\begin{equation}\label{eq:pfd}
\frac{\prod_{j=1}^{n-1}(z- b_j)}{\prod_{j=1}^{n} (z-a_j)} = \sum_{j=1}^{n} \frac{p_j}{z-a_j}~.
\end{equation}
Vice versa, a probability measure supported on a finite set of
atoms gives rise to a pair of interlacing sequences.

This construction admits numerous generalisations. The relation
\begin{equation}\label{eq:pfd'}
\exp \left\{ -  \int \log(z- x) d (\tau_+(x)-\tau_-(x)) \right\} = \int \frac{d\mu(x)}{z-x}~,
\end{equation}
obtained from  (\ref{eq:pfd}) by replacing sums with integrals, forms the basis of the solution of the Markov moment problem (see \cite{M1,AK_book,KN}), and is
one of the forms of the Markov correspondence (which we further discuss below and in \ref{hom}). In the terminology of
Kerov \cite{Kerov_book,Kerov_interl},  $\tau_+$ and $\tau_-$ corresponding to a probability measure $\mu$  form a pair of interlacing measures; such pairs are intrinsically characterised by the inequalities
\begin{equation}\label{eq:interlm}
\tau_+[x, \infty) \geq \tau_-[x, \infty)~, \quad \tau_+(-\infty, x] \geq \tau_-(-\infty, x] \quad (x \in \mathbb{R})~.
\end{equation}

The equality (\ref{eq:pfd'})  may  be viewed as
a connection between an additive and a multiplicative representation of a
function from the Nevanlinna class. In this form it admits further generalisations,
extensively used starting from the works of Akhiezer and Krein (see e.g.\ the appendix to \cite{KN}); we also refer to the work of Yuditskii
\cite{Yud} for some recent developments.

\medskip\noindent
In the 1990-s, Kerov discovered a number of applications of the Markov correspondence
to problems in representation theory and analysis. These applications form
a central theme in the monograph \cite{Kerov_book}; see further the survey \cite{Kerov_interl}.

Here we follow Kerov and switch to the language of continual diagrams, which are $1$--Lipschitz functions coinciding with $|x-a|$ for large values of $|x|$ (such as in Figure~\ref{fig:young} [right] and Figure~\ref{fig}; see \ref{contdiag} for a
formal definition).
The mapping from pairs of interlacing sequences to continual diagrams is given by
$(A,B) \mapsto \omega$,
\begin{equation}\label{eq:omega} \omega(x) = \sum_{j=1}^n |x - a_j| - \sum_{j=1}^{n-1} |x-b_j|~;\end{equation}
whereas for interlacing measures $(\tau_+, \tau_-)$ as in (\ref{eq:pfd'}) or (\ref{eq:interlm}), one sets
\[ \omega(x) = \int |x-a| \, d\tau_+(a) - \int |x-b| \, d\tau_-(b)~. \]
The Markov correspondence  (\ref{eq:pfd'}) induces a bijection between
continual diagrams $\omega$ and probability measures $\mu$, which is 
called the Markov transform and denoted $\mu = \Mar \omega$. Some of its properties are listed in \ref{hom}.

\medskip\noindent
Continual diagrams appear naturally as scaling limits of Young diagrams. Indeed,
a Young diagram rotated by $135^\circ$ (with respect to the English convention) is a continual diagram, see
Figure~\ref{fig:young} for an illustration and \ref{young}, \ref{plan} for the formal construction.
\begin{figure}
\begin{center}
\yng(7,4,4,3,1)
\,\,\,
\YRussian
 \yng(7,4,4,3,1)
 \,\,
\raisebox{-1cm}{\begin{tikzpicture}
      \draw[scale=0.3,->] (-10,0) -- (10,0) ;
      \draw[scale=0.3,->] (0,0) -- (0,10);
      \draw[scale=0.3,domain=-8:9,samples=100,variable=\x,blue,thick] plot ({\x},{abs(\x-7)+abs(\x-3)+abs(\x)+abs(\x+3)+abs(\x+5)-abs(\x-6) - abs(\x-1)-abs(\x+1) -abs(\x+4)});
      \draw[scale=0.2,domain=-10:11,samples=200,variable=\x,red,dashed] plot ({\x},{abs(\x)});    \end{tikzpicture}}
\end{center}
\caption{A Young diagram in English convention (left) and in Russian convention (centre), and the corresponding continual diagram (right).}\label{fig:young}
\end{figure}
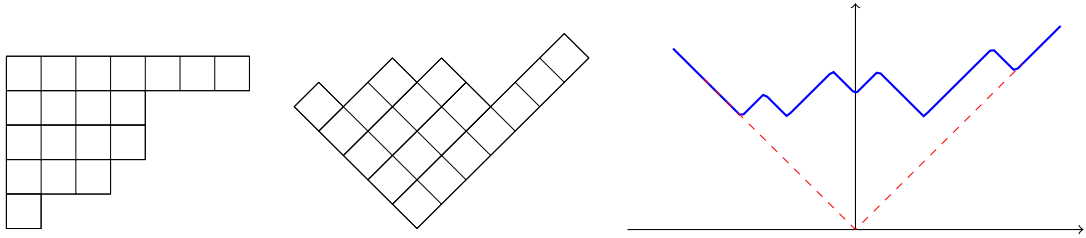
Kerov showed that the Markov transform $\mu_n = \Mar \omega_n$ of a continual
diagram $\omega_n$ obtained in this way encodes the transition probabilities of
the Young diagram in a stochastic process called the Plancherel growth (see \ref{tran}), and
called $\mu_n$ the transition measure of the Young diagram. If $\omega_n$ is
a random continual diagram associated with a Young
diagram  sampled at random from the Plancherel measure, the Logan--Shepp--Vershik--Kerov limit law \cite{LS,VK1,VK2} asserts that (uniformly almost surely)
\begin{equation}\label{eq:lsvk0} \frac{1}{\sqrt{n}} \omega_n(\sqrt{n} x) \longrightarrow \Omega_\LSVK(x) =
\begin{cases}
\frac2\pi \left(x \arcsin\frac{x}{2} + \sqrt{4-x^2} \right)~, & |x| \leq 2 \\
|x|~, & |x| > 2
\end{cases}\end{equation}
Using the Markov correspondence, Kerov deduced that $\tilde\mu_n$ (here and further
in the introduction,  tildes indicate unspecified scaling) obey the semicircle law 
\cite{Kerov_tr}
\begin{equation}\label{eq:k0} 
d\tilde\mu_n(x) \to d\rho_\sc(x) = \frac{1}{2\pi}\sqrt{(4-x^2)_+} \,dx~;
\end{equation}
vice versa, (\ref{eq:k0}) implies  (\ref{eq:lsvk0}).

\smallskip\noindent
Kerov also discovered  \cite{Kerov_sep} a random matrix counterpart of these
statements; it comes in two flavours. Let $\tilde{P}_n(x)$ be the characteristic polynomial of
a Wigner matrix $\tilde H_n$ of dimension $n \times n$ (see \ref{wig1}); let $\tilde\lambda_j$ be the zeros of
$\tilde{P}_n$ (which are the eigenvalues of $\tilde H_n$). Also let $\tilde\lambda_j^*$ be the
zeros of $\tilde P_n'$, and let $\tilde\lambda_j^{**}$ be the eigenvalues of the top-left $(n-1) \times (n-1)$ principal submatrix of $\tilde H_n$. Then (see \ref{wig1})
\begin{equation}\label{eq:lsvk_glob}\sum_{j=1}^{n} |x - \tilde\lambda_j| - \sum_{j=1}^{n-1} |x - \tilde\lambda_j^*| \to \Omega_\LSVK(x) \end{equation}
and also (see \ref{kb2})
\begin{equation}\label{eq:lsvk_loc} \sum_{j=1}^{n} |x - \tilde\lambda_j| - \sum_{j=1}^{n-1} |x - \tilde\lambda_j^{**}| \to \Omega_\LSVK(x)~. \end{equation}
The former was proved by Kerov; the latter was proved by Alexey Bufetov \cite{Buf}
who strengthened a  result in \cite{Kerov_sep}.

As put forth by Kerov, (\ref{eq:lsvk_glob}) is equivalent (by the Markov correspondence) to Wigner's semicircle  law
\[ \widetilde{\rho}_n= \frac{1}{n} \sum_{j=1}^{n} \delta_{\tilde\lambda_j} \to \rho_\sc\]
for the normalised eigenvalue counting measure, see further \ref{kb2}. On the other hand, (\ref{eq:lsvk_loc})
is equivalent to Wigner's law for the spectral measure associated with a fixed
vector in $\mathbb C^{n}$, see \ref{wig1}.

\smallskip\noindent
Another similar looking result was found by Kerov \cite{Kerov_sep} in the setting
of Jacobi matrices with regularly varying coefficients. There (see \ref{jac}), the counterparts of $\tilde\lambda_j$ and $\tilde\lambda_{j}^{**}$ are the (properly rescaled) eigenvalues of
an $n\times n$ and an $(n-1) \times (n-1)$ principal submatrix, respectively; then (\ref{eq:lsvk_loc}) holds, as a consequence of the semicircle limit law for what
could be colloquially called the spectral measure at infinity (see \ref{jac} and \ref{limjac_1}). On the other hand, the counterpart of
(\ref{eq:lsvk_glob}) is, in general, false (see \ref{arcsin}).

\medskip\noindent
In Section~\ref{S:lim}, mostly following \cite{Kerov_book,Kerov_interl} in substance if not
in terminology, we give an overview of these results with some proofs.

\subsection{Fluctuations about the limiting shape}\label{intr:fl}

In the main part of the paper (Sections~\ref{S:cor} and \ref{S:part}), we use the Markov correspondence to study the
deviations of diagrams and of interlacing sequences from the limiting shape. We observe that, although
(\ref{eq:pfd'}) is highly non-linear, it can be linearised about the limiting shape. Therefore
one can study the fluctuations of the left-hand side via the right-hand side, and vice versa.
Several forms of this assertion are proved in Section~\ref{s:detcor}, for example, in
the case of the limiting shape $\Omega_\LSVK$ we have:
\begin{namedprop}[\ref{pi}]
Let $\mu_n$ be  probability measures, and let $\omega_n = \Mar \mu_n$ be the corresponding continual diagrams.
Let $(\epsilon_n)_n$  and $(\alpha_k)_{k \geq 2}$ be two  sequences such that
$\epsilon_n \to 0$ and $ \frac1k \log |\alpha_k| \to 0$.
Then
\begin{equation}\label{eq:pi_assume} \lim_{n \to \infty} \int_{-2}^2 \phi(x) \frac{d\mu_n(x) - d\rho_\sc(x)}{\epsilon_n} =  \sum_{k\geq 2} \alpha_k  \int_{-2}^2 \phi(x) \,2 T_k(x/2)\, d\rho_\arcsin(x) 
\end{equation}
for all test functions $\phi$ which are analytic in  $[-2,2]$ if and only if
\begin{equation}\label{eq:pi_conclude} \lim_{n \to \infty}  \int_{-2}^2 \phi(x)  \frac{\omega_n(x) - \Omega_\LSVK(x)}{\epsilon_n} dx =  \sum_{k \geq 1} \frac{4 \alpha_{k+2}}{k+1} \int_{-2}^2 \phi(x) U_{k}(x/2) d\rho_\sc(x) \end{equation}
for all such $\phi$. \end{namedprop}
\noindent Here
\[ d\rho_\arcsin(x) = \frac{\mathbbm{1}_{[-2,2]}(x)dx}{\pi\sqrt{4-x^2}} \]
is the arcsine distribution,  $T_k$ and $U_k$ are the Chebyshev polynomials (\ref{eq:chebdef}), and the Markov transform $\mathfrak M$ is formally defined in \ref{hom}. A marginally more general formulation is given in \ref{pi}, and the
stochastic setting (as opposed to deterministic deviations from the limit shape) is commented upon
in \ref{clt}. Note that the integrals are exactly
the  coefficients of $\phi$ in the Chebyshev expansions:
\[\begin{split} \phi(x) &= \sum_{k \geq 0} T_k(x/2)
 \int_{-2}^2 \phi(y) \, (2-\delta_{k0}) T_k(y/2)\,  d\rho_\arcsin(y) \\
 &=
 \sum_{k \geq 0} U_k(x/2) \int_{-2}^2 \phi(y) U_{k}(y/2) d\rho_\sc(y)~.
 \end{split}\]
In particular, the assumption (\ref{eq:pi_assume}) implies that (and is almost equivalent to)
 \begin{equation}\label{eq:termwise} \int T_k(x/2) (d\mu_n(x) - d\rho_\sc(x)) = \epsilon_n \alpha_k + o(\epsilon_n)~, \quad n \to \infty~.\end{equation}
 We use Proposition~\ref{pi} to study the fluctuations of random  diagrams
 appearing in the theory of random matrices and in the representation theory
 of the symmetric group.

\medskip\noindent
The study of fluctuations of diagrams was initiated by a theorem, proved by Kerov \cite{Kerov_CLT} in the 1990-s, which describes the fluctuations of a (Plancherel) random Young diagram about the limiting shape. Informally,
\begin{equation}\label{eq:clt_kerov} \frac{1}{\sqrt{n}} \omega_n(\sqrt{n}x)  \approx \Omega_\LSVK(x) + \frac{1}{\sqrt{n}}  \sum_{k \geq 1} \frac{2 g_{k+2}}{\sqrt{k+1}} \, \frac{U_{k}(x/2) \sqrt{4-x^2}}{2\pi}~,\end{equation}
where $g_k$ are independent, identically distributed standard Gaussian variables
(see \ref{kerov_clt}). Another proof, based on Kerov's unpublished notes, was given
by Ivanov and Olshanski in \cite{IO}. We refer to the works \cite{BufGor,Mel,Moll}
for various generalisations, not discussed here.

Ivanov and Olshanski also described the fluctuations of the transition measure $\tilde\mu_n$ associated with $\omega_n$ (see \ref{tran}):
\begin{equation}\label{eq:clt_io} d\tilde\mu_n(x) \approx d\rho_\sc(x) + \frac{1}{\sqrt{n}} \sum_{k \geq 3} \frac{\sqrt{k-1}}{2} g_{k} \, \frac{2 T_k(x/2)dx}{\pi \sqrt{4-x^2}}~. \end{equation}

As observed in \cite{IO}, (\ref{eq:clt_io}) bears a similarity to Johansson's central limit
theorem \cite{Joh} for the Gaussian Unitary Ensemble:
\begin{equation}\label{eq:joh_clt}
 d\tilde\rho_n(x) \approx d\rho_\sc(x) + \frac{1}{n}  \sum_{k \geq 1} \frac{\sqrt{k}}2 g_k \, \frac{2 T_k(x/2)dx}{\pi\sqrt{4-x^2}}
 \end{equation}
 (see \ref{kerov_clt}). This raises the question what is the analogue of Kerov's central limit theorem (\ref{eq:clt_kerov}) in random matrix context.

\bigskip\noindent
Recently, this question was studied by Erd\H{o}s and Schr\"oder  \cite{ES}. Their
result is both general (generalised Wigner matrices are considered), and strong
(the $\approx$ sign in (\ref{eq:es}) below can be understood pointwise, with an explicit
power-law estimate on the error term). In the special case of the Gaussian Unitary Ensemble, the result of \cite{ES} asserts that
the fluctuations of the diagram corresponding to the eigenvalues of $\tilde H_n$ and of its principal submatrix are described by
\begin{align}\label{eq:es}
&\sum_{j=1}^{n+1} |x - \tilde\lambda_j| - \sum_{j=1}^{n} |x - \tilde\lambda_j^{**}|  \approx \Omega_{LSVK}(x) + \frac{1}{\sqrt{n}} \Delta_{1}^{\Mar}(x)~, \\
\label{eq:es'}
&\Delta_{1}^\Mar(x) =  \frac{2g_1}{\pi} \arcsin \frac{x}{2} +
\sum_{k\geq 1} \frac{2(g_{k}-g_{k+2})}{k+1} \frac{U_k(x/2) \sqrt{4-x^2}}{2\pi}~,
\end{align}
where the argument of the arcsine is truncated at $\pm 1$ (see further \ref{lp} and \ref{lp2ors}). Parallel results for Jacobi $\beta$-ensembles
were obtained by Gorin and Zhang in \cite{GZ}.

The right-hand side of (\ref{eq:es}) does not look similar to (\ref{eq:clt_kerov}),
for two reasons. First, a typical random partition of $n$ has $\asymp \sqrt{n}$
rows, therefore the normalisation $1/\sqrt{n}$ in (\ref{eq:es}) would correspond to
$1/\sqrt[4]{n}$, rather than $1/\sqrt{n}$, in (\ref{eq:clt_kerov}).
Second, the Gaussian process in (\ref{eq:es'})  has a continuous modification (it is roughly a reparametrised Brownian motion), whereas the sum in (\ref{eq:clt_kerov}) does not
even converge in $L_2$.

\medskip\noindent
These differences are natural in view of the Markov correspondence.  As mentioned in Section~\ref{s:mc} 
above, the Markov transform takes the left-hand side of (\ref{eq:es})
 to the spectral measure $\tilde\mu_n$ (see \ref{wig1}). The fluctuations of the latter were studied, in the context
 of Gaussian ensembles, by Lytova and Pastur \cite{LP}; for the GUE, their result (see \ref{lp}) asserts that
\begin{align}\label{eq:lp}
&d\tilde\mu_n(x) \approx d\rho_\sc(x) + \frac{1}{\sqrt{n}} \Delta_{1}(x)dx~, \\
&\Delta_{1}(x) \sim \sum_{k \geq 1} g_k U_k(x/2) \frac{\sqrt{4-x^2}}{2\pi}~.
\end{align}
As noted in \cite{LP}, the larger scale $1/\sqrt{n}$ reflects the Gaussian fluctuations
of the eigenvectors of the random matrix.
The result (\ref{eq:lp}) was extended to other Wigner matrices by the same authors
\cite{LP2} and by
Pizzo--Renfrew--Sosh\-nikov \cite{ORS}; in this more general setting,
 the limiting process
may change, and is not necessarily Gaussian (see \ref{lp2ors}).

In \ref{lp}, we feed (\ref{eq:lp}) into the Markov correspondence and obtain
a version of (\ref{eq:es}), albeit in much weaker topology than \cite{ES}.
More general Wigner matrices are considered in  \ref{lp2ors}.
(Vice versa, it may be possible to deduce a strengthened version of the central limit
theorems in \cite{LP2,ORS} from the result of Erd\H{o}s--Schr\"oder
\cite{ES} in its full strength and generality; we do not pursue
this direction here.\footnote{See further \ref{esfluct}.})

\medskip\noindent
The discussion above raises the expectation that a version of (\ref{eq:es}) with
the critical points $\tilde\lambda_j^{*}$ in place of the submatrix eigenvalues $\tilde\lambda_j^{**}$ may bear more similarity to Kerov's theorem
(\ref{eq:clt_kerov}). Indeed, appealing to Johansson's central limit theorem (\ref{eq:joh_clt}), we show
that
\begin{namedcor}[\ref{joh}] For the Gaussian Unitary Ensemble,
\begin{align}\label{eq:es_glob}
&\sum_{j=1}^{n} |x - \tilde\lambda_j| - \sum_{j=1}^{n-1} |x - \tilde\lambda_j^{*}|  \approx \Omega_{LSVK}(x) + \frac{1}{n} \Delta_{\tr}^{\Mar}(x)~, \\
\label{eq:es_glob'}
&\Delta_{\tr}^{\Mar}(x) \sim  \frac{2}{\pi} g_1 \arcsin(x/2) + \sum_{k\geq 0} \frac{2 \sqrt{k+2}}{k+1} g_{k+2} \, \frac{U_{k}(x/2)\sqrt{4-x^2}}{2\pi}
\end{align}
\end{namedcor}
The precise formulation is given in \ref{joh},  a comparison between (\ref{eq:es_glob}) and (\ref{eq:es}) --- in Figure~\ref{fig}, and a generalisation to other Wigner matrices ---
in \ref{joh_gen}.

\begin{figure}\label{fig}
\begin{center}
\includegraphics[trim={1cm 4cm  1cm 4cm}, scale=.4]{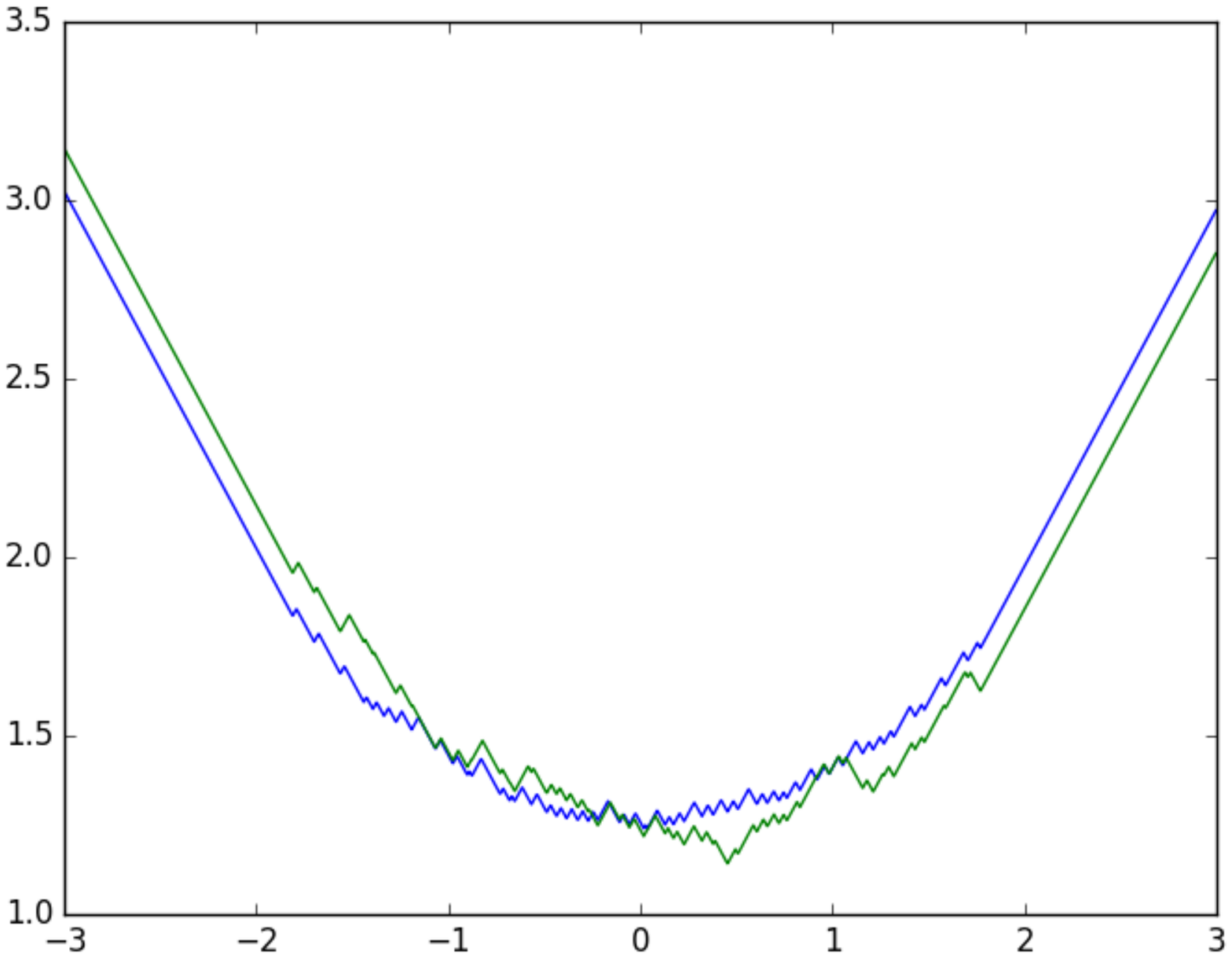}
\end{center}
\caption{The random continual diagrams (\ref{eq:es_glob}) in blue and (\ref{eq:es}) in green, for GUE$_{n=50}$. The former 
fluctuates on scale $\sim 1/n$, while the latter --- on scale $1/\sqrt{n}$.}
\end{figure}

\bigskip\noindent
Recently,  Fyodorov asked what are the properties of the critical points $\tilde\lambda_j^*$,
and in particular how to distinguish between them and $\tilde\lambda_j^{**}$. Differentiating the relations
(\ref{eq:es_glob})--(\ref{eq:es_glob'}) and comparing to (\ref{eq:es})--(\ref{eq:es'}), we
obtain an answer in the following form (stated here for GUE; the topology is as in Proposition~\ref{pi}):
\begin{align}
\label{eq:es_glob''}
&\sum_{j=1}^{n} \delta_{\tilde\lambda_j} - \sum_{j=1}^{n-1} \delta_{\tilde\lambda_j^{*}} \approx
\frac{dx}{\pi\sqrt{4-x^2}} + \frac{1}{2n} \left[ \frac{d^2}{dx^2} \Delta_\tr^\Mar (x)\right]dx
\\ \label{eq:es_loc''}
&\sum_{j=1}^{n} \delta_{\tilde\lambda_j} - \sum_{j=1}^{n-1} \delta_{\tilde\lambda_j^{**}} \approx
\frac{dx}{\pi\sqrt{4-x^2}} + \frac{1}{2\sqrt n} \left[\frac{d^2}{dx^2} \Delta_{1}^\Mar(x) \right]dx~;
\end{align}
note the difference in scaling, and that in the former, the Gaussian process is a derivative of a log-correlated process,
while in the latter, it is a derivative of (reparametrised) white noise (cf.\ (\ref{eq:lp_ito})).
The relation (\ref{eq:es_loc''}) was proved by Erd\H{o}s and Schr\"oder, in greater
generality (Wigner matrices) and stronger topology (corresponding to test functions in the Sobolev space $H^2[-10,10]$); they  used it to prove
(\ref{eq:es}). The relation (\ref{eq:es_glob''}) seems not to have been observed before.
The formul{\ae} (\ref{eq:es_glob''}) and (\ref{eq:es_loc''}) appear as (\ref{eq:diffcrit})
and (\ref{eq:diffsub}), respectively, in the body of the paper.

The comparison between the statistical properties of the critical points and the eigenvalues in the local regime (i.e. on scales 
comparable to the mean spacing) is discussed in the companion
paper \cite{me-crit}.

\bigskip\noindent
In Section~\ref{S:part}, we return to random partitions and to the theorems
 of Kerov and Ivanov--Olshanski, stated above as (\ref{eq:clt_kerov}) and (\ref{eq:clt_io}). We describe the
setting, and use Proposition~\ref{pi} to derive one from the other.

In Section~\ref{s:pfs} we outline a
proof of (\ref{eq:clt_io}) (and thus also of (\ref{eq:clt_kerov})) based on the combinatorial
approach of Biane \cite{Bi} and Okounkov \cite{Ok} in the version of \cite{JS}. Our goal is to emphasise
the similarity between the transition measure of a random diagram and the normalised
eigenvalue counting measure of a random matrix. To this end, we compare the Jucys--Murphy elements (see \ref{repr}) acting on a random representation of the symmetric group
with Wigner random matrices chosen from an ensemble (see \ref{unimod})
for which  a particularly clean version of the moment method is available
(cf.\ \cite{FeldS,me_surv} and references therein).
For this ensemble,
\begin{equation}\label{eq:joh_clt'}
 d\tilde\rho_n(x) \approx d\rho_\sc(x) + \frac{1}{n}  \sum_{k \geq 3} \frac{\sqrt{k}}2 g_k \, \frac{2 T_k(x/2)dx}{\pi\sqrt{4-x^2}}
 \end{equation}
 and accordingly (Proposition~\ref{unimod})
\begin{equation}\begin{split}
&\sum_{j=1}^n |x - \tilde\lambda_j|  -\sum_{j=1}^{n-1} |x - \tilde\lambda_j^{*}| \\
&\qquad\approx  \Omega_\LSVK(x) +
 \frac1n \sum_{k \geq 1} \frac{2\sqrt{k+2}}{k+1} g_{k+2} \frac{U_k(x/2)\sqrt{4-x^2}}{2\pi}~.
\end{split}\end{equation}
In the combinatorial approach, the coefficient $k$ (the square root of which appears
in (\ref{eq:joh_clt'})) acquires the interpretation as the number of ways to align
two cycles of length $k$. The coefficient $k-1$ in (\ref{eq:clt_io}) has exactly the
same combinatorial meaning (with cycles of length $k-1$).

\bigskip\noindent

Many of the results are probably familiar to experts; however, we hope to find a reader
that would enjoy seeing them under a single cover, rephrased in the peculiar argot
of spectral theory. We made no attempt to pursue the limits of the approach, and instead chose to
illustrate the main ideas in the simplest setting.

\section{Limit shape}\label{S:lim}

\subsection{Jacobi matrices}

\subsubsection{}\label{jac}

Let $J$ be a Jacobi matrix and let $J_n$ be its top-left $n \times n$
principal submatrix,
\begin{equation}\label{eq:defJ}
J = \left( \begin{array}{ccccc}
a_1 & b_1 & 0 & 0 &\cdots  \\
b_1 & a_2 & b_2 & 0 & \cdots \\
0 & b_2 & a_3 & b_3 & \cdots  \\
\cdots &\cdots&\cdots&\cdots&\cdots
\end{array}\right) \,\,\,
J_n = \left( \begin{array}{ccccc}
a_1 & b_1 & 0 & 0 &\cdots  \\
b_1 & a_2 & b_2 & 0 & \cdots \\
\cdots &\cdots&\cdots&\cdots&\cdots\\
0 & 0  & \cdots & b_{n-1} & a_n
\end{array}\right)\end{equation}
Define a probability measure $\mu_n$ by
\begin{equation}\label{eq:sp} \int x^k d\mu_n(x) = (J_n^k)_{nn} = \langle J_n^k \delta_n, \delta_n \rangle~, \quad k = 0, 1, 2, \cdots \end{equation}
It may be called the spectral measure of $J_n$ at $\delta_n$, or, following Kerov,
the \mbox{$n$-th} transition measure of $J$. It is supported on the eigenvalues of $J_n$; the
mass at an eigenvalue is equal to the squared $n$-th coordinate of the corresponding eigenvector.

Also define the normalised eigenvalue counting measure $\rho_n$ by
\begin{equation}\label{eq:ecm} \int x^k d\rho_n(x) = \frac{1}{n} \tr J_n^k~, \quad k = 0, 1, 2, \cdots \end{equation}
which has equal masses $1/n$ at the eigenvalues of $J_n$.

\subsubsection{}\label{limjac_1}

The asymptotics of $\mu_n$ (colloquially, the spectral measure at infinity) is determined by the asymptotics of $a_n$ and $b_n$. The
following proposition is folklore (cf.\  Kerov \cite{Kerov_sep}):
\begin{prop*} Let $J$ be a Jacobi matrix (\ref{eq:defJ}) such that
\begin{equation}\label{eq:regvar}
\lim_{n \to \infty} \frac{b_{n-1}}{b_n} = 1~, \quad \lim_{n \to \infty} \frac{a_{n-1}-a_n}{b_n} = 0~.\end{equation}
Then the sequence of measures $\widetilde\mu_n$ defined by
\[ \widetilde\mu_n(B) = \mu_n(b_{n-1} B + a_n)~, \quad B \in \mathcal B(\mathbb R)~, \]
converges weakly to the semicircle measure $\rho_\sc$,
\[ \rho_\sc(B) = \int\limits_{B \cap (-2, 2)} \frac1{2\pi} \sqrt{4-x^2}\, dx~.\]
\end{prop*}

\begin{proof} The measure $\widetilde{\mu}_n$  is the spectral measure of
$\widetilde{J}_n$ at $\delta_1$, where
\[ (\widetilde J_n)_{i\,j} = b_{n-1}^{-1} \left( (J_n)_{n+1-i\,n+1-j} - a_n \right)~. \]
As $n \to \infty$,
\[ b_{n-1}^{-1} (\widetilde{J}_n - a_n) \to T \]
in strong operator topology, where $T_{ij} = \delta_{i,j+1} + \delta_{i+1,j}$ is
the adjacency matrix of $\mathbb{Z}_{>0}$. The spectral measure
of $T$ at $\delta_1$ is exactly $\rho_\sc$, as one can see, for example,
from the relation
\[ ((T-z)^{-1})_{11} = (- z - ((T-z)^{-1})_{11})^{-1}~, \quad z \in \mathbb{C}\setminus\mathbb{R}~,\]
which follows from the formula for the top-left matrix element of a matrix inverse.
Therefore $\widetilde{\mu}_n \to \rho_\sc$.
\end{proof}

\subsubsection{}
The previous discussion remains valid for a sequence of Jacobi matrices
which are not necessarily  the sub-matrices of one infinite matrix. By the same
argument as above, we have:

\begin{prop*} Let $(J_n)_{n \geq 0}$ be sequence of finite
Jacobi matrices
\[J_n = \left( \begin{array}{ccccc}
a_{1,n} & b_{1,n} & 0 & 0 &\cdots  \\
b_{1,n} & a_{2,n} & b_{2,n} & 0 & \cdots \\
\cdots &\cdots&\cdots&\cdots&\cdots\\
0 & 0  & \cdots & b_{n-1,n} & a_{n,n}
\end{array}\right) \]
such  that for any $k \geq 1$
\begin{equation}\label{eq:regvar'}
\lim_{n \to \infty} \frac{b_{n-k-1,n}}{b_{n-1,n}} = 1~, \quad \lim_{n \to \infty} \frac{a_{n-k,n}-a_{n,n}}{b_{n-1,n}} = 0~.\end{equation}
Then the sequence of measures $\widetilde\mu_n(B) = \mu_n(b_{n-1,n} B + a_{n,n})$
converges weakly to the semicircle measure $\rho_\sc$.
\end{prop*}

\subsubsection{}\label{rhon}

The asymptotics of $\rho_n$ is also determined by the asymptotics of $a_n$ and $b_n$.
However, $\mu_n$
is insensitive to multiplication of the coefficients by a sequence $c_n$ such that
$\lim \frac{c_n}{c_{n-1}} = 1$, whereas $\rho_n$ depends on the growth of the coefficients. For example, we have

\begin{prop*} Suppose
\begin{equation}\label{eq:herm}
a_n = o(\sqrt{n})~, \quad b_n = \sqrt{n}(1 + o(1))~.
\end{equation}
Then
the rescaled measures
$\widetilde{\rho}_n(B) = \rho_n(\sqrt{n} B)$ converge weakly to $\rho_\sc$.
\end{prop*}

\begin{proof} Let $a_{n}^{\mathrm{H}}=  0$,
$b_n^{\mathrm{H}} = \sqrt{n}$. Here $\mathrm H$ stands for Hermite, and $\rho_n^{\mathrm H}$ is the normalised
zero counting measure of Hermite polynomials. It is known \cite{Sz} that $\widetilde{\rho}_n^{\mathrm H} \to \rho_\sc$. This can also be proved directly from the Jacobi
matrix as in \ref{limjac_1}.

To justify the approximation $\rho_n \approx \rho_n^{\mathrm H}$, let $\epsilon > 0$. Choose $n_0$ such that for $n > n_0$
\begin{equation}\label{eq:norm}
|a_n| < \epsilon \sqrt n~, \quad |b_n - \sqrt{n}| < \epsilon \sqrt n~.
\end{equation}
Let
\[ a_n^1 = \begin{cases} 0~, &n \leq n_0 \\ a_n~, &n > n_0 \end{cases}~, \quad
   \begin{cases} b_n^1 = \sqrt{n}~, &n \leq n_0 \\ b_n~, &n > n_0 \end{cases}~.\]
 Then, for any segment $[a, b]$,
 \[ \left| \widetilde\rho_n[a,b] - \widetilde\rho_n^1[a,b]\right| \leq \frac{n_0}{n} \]
 by the interlacing property of rank-one perturbation, and
\[
\widetilde\rho_n^{\mathrm H}[a+\epsilon,b-\epsilon] \leq \widetilde\rho_n^1[a,b] \leq \widetilde\rho_n^{\mathrm H}[a-\epsilon,b+\epsilon]~. \]
by (\ref{eq:norm}). It remains to let $n \to \infty$ and then $\epsilon \to +0$.\end{proof}

\subsubsection{}\label{arcsin}
In the case of Proposition~\ref{rhon} the sequences $\rho_n$ and $\mu_n$ share the same asymptotics. As emphasised by Kerov, this is an exception rather than a rule. The limit of the
former is a kind of integrated density of states, while the latter describes the spectral properties at infinity. Neither is directly related to the usual spectral measure at $\delta_1$.

\begin{ex*}
For $J = T$, $\rho_n \to \rho_{\mathrm{arcsin}}$, where
\[ \rho_{\mathrm{arcsin}}(B) = \int_{B \cap (-2,2)}\frac1\pi  \frac{dx}{\sqrt{4-x^2}}~.\]
The same conclusion holds for any Jacobi matrix with
\[ a_n = o(1)~, \quad b_n = 1 + o(1)~.\]
\end{ex*}

\subsection{Continual diagrams}

\subsubsection{}\label{contdiag} A continual diagram is a function $\omega: \mathbb{R} \to \mathbb{R}$
such that for some $a \in \mathbb{R}$
\begin{align}
&|\omega(x) - \omega(y)| \leq |x-y|~, \\
&\omega(x) = |x - a|
\quad \text{for sufficiently large $x$} \label{eq:sufflarge}
\end{align}

The collection of diagrams is equipped with the topology of uniform convergence. A diagram $\omega$ is said to be supported in a (closed) segment $I$ (denoted: $\omega \in \mathcal D(I)$) if
(\ref{eq:sufflarge}) holds for all $x \notin I$.

\subsubsection{}\label{hom}

Denote by $\mathcal M(I)$ the collection of Borel probability measures on $I$,
equipped with weak topology.

\begin{thm*}[Markov \cite{M1}; Akhiezer--Krein \cite{AK1}; Kerov \cite{Kerov_tr}] For any segment $I$, the relation
\[ \exp \left\{ - \frac12  \int \log(1- zx) d \omega'(x) \right\} = \int \frac{d\mu(x)}{1-zx} \]
defines a homeomorphism $\Mar:\mathcal{D}(I) \longleftrightarrow \mathcal{M}(I)$.
\end{thm*}
The homeomorphism $\Mar$ is called the Markov transform.
In the language of what is now called the Markov moment problem, the construction of the bijection
$\Mar$  for the case of a segment goes back to Markov
\cite{M1}. It was developed and generalised by Akhiezer and Krein in the 1930-s,
who published a series of papers \cite{AK0,AK0.5,AK1,AK2} and a book \cite{AK_book}
on this subject.
The formulation in the language of continual diagrams is due to Kerov \cite{Kerov_tr}, who also observed that $\Mar$ is a homeomorphism;
see further \cite{Kerov_book}.

\begin{ex*} The Logan--Shepp--Vershik--Kerov diagram
\[ \Omega_\LSVK(x) = \begin{cases}
\frac2\pi \left(x \arcsin\frac{x}{2} + \sqrt{4-x^2} \right)~, & |x| \leq 2 \\
|x|~, & |x| > 2
\end{cases}~, \]
corresponds to the semicircle: $\Mar \, \Omega_\LSVK = \rho_\sc$.
\end{ex*}

\subsubsection{}\label{jacdiag}

Let $J$ be a Jacobi matrix. Denote by $\lambda_j^{(n)}$ the eigenvalues of $J_n$,
and define the diagram $\omega_n$ corresponding to the interlacing sequences
$(\lambda_j^{(n)})$ and $(\lambda_j^{(n-1)})$ via (\ref{eq:omega}), i.e.\ as a continuous function such that
\begin{empheq}[left = \empheqlbrace]{align}
\omega_n'(x) &= -1 + 2\, \# \left\{j \, , \, \lambda_j^{(n)} \leq x \right\} - 2\,  \# \left\{j \, , \, \lambda_j^{(n-1)} \leq x \right\} \nonumber\\
 \omega_n(x) &= |x - \sum_{j=1}^{n} \lambda_j^{(n)} + \sum_{j=1}^{n-1} \lambda_j^{(n-1)} |
\quad \text{for sufficiently large $x$}.\nonumber
\end{empheq}
Equivalently, if $P_n(z) = \det (z - J_n)$,
\[ \omega_n'(x) = \operatorname{sign} \frac{P_{n-1}(x)}{P_n(x)}~. \]
Also define a diagram $\varpi_n$ such that
\[ \varpi_n'(x) = \operatorname{sign} \frac{P_{n}'(x)}{P_n(x)}
 = \operatorname{sign} \frac1{n} \frac{P_{n}'(x)}{P_n(x)}~. \]

\begin{lemma*} $\Mar  \omega_n = \mu_n$ and $\Mar \varpi_n = \rho_n$.
\end{lemma*}

\begin{proof} Representing $P_{n-1}/P_n$ and $P_n'/P_n$ as a sum
of  simple fractions, we obtain:
\[ \frac{P_{n-1}(z)}{P_n(z)} = \int \frac{d\mu_n(x)}{z-x}~, \quad
\frac{1}{n}\frac{P_{n}'(z)}{P_n(z)} = \int \frac{d\rho_n(x)}{z-x}\]
and then use the bijection $\mathfrak M$ from Theorem~\ref{hom}.
\end{proof}

As noted by Kerov (e.g.\ \cite[Section 6]{Kerov_interl}), this lemma is a finite-dimensional
trace formula (the study of trace formul{\ae} goes back to the works of Lifshits \cite{Lif}
 and Krein \cite{Krein_trace}, see further the survey of Birman and Yafaev \cite{BY}).

\subsubsection{}\label{limjac_2}

\begin{cor*}[Kerov \cite{Kerov_sep}] Let $J$ be a Jacobi matrix (\ref{eq:defJ}) satisfying (\ref{eq:regvar}).
Then
\[ b_{n-1}^{-1} \omega_n(b_{n-1} x + a_n) \to \Omega_\LSVK(x) \]
uniformly in $x$.
\end{cor*}

\begin{proof} Follows from Proposition~\ref{limjac_1}, Lemma~\ref{jacdiag} and Theorem~\ref{hom}.
\end{proof}

\subsubsection{}\label{limjac_3}

The corresponding statement for $\varpi_n = \Mar^{-1} \rho_n$ is much less general.

\begin{cor*}[Kerov \cite{Kerov_sep}] Let $J$ be a Jacobi matrix (\ref{eq:defJ}) satisfying (\ref{eq:herm}).
Then
\[ \frac{1}{\sqrt{n}} \varpi_n(\sqrt{n}x ) \to \Omega_\LSVK(x) \]
uniformly in $x$.
\end{cor*}

\begin{proof} Follows from Proposition~\ref{rhon}, Lemma~\ref{jacdiag} and Theorem~\ref{hom}.
\end{proof}

\subsubsection{}\label{ed}

Introduce the rescaling operators
\[ (\Resc_L \mu)(B) = \mu(L B) \quad (B \in \mathcal B(\mathbb R))~, \quad (\Resc_L \omega)(x) = \frac1L \omega(Lx)~.\]

Consider the following random Jacobi matrix,
constructed by Dumitriu and Edelman \cite{DE}.
Fix $\beta > 0$, and let $J^\beta$ be such that
\[ a_n^{\beta} \sim N(0, 2/\beta)~, \quad
\beta \, b_n^\beta \sim \chi_{n\beta}~,\]
 and these random variables are jointly independent.
Recall that the $\chi_a$ distribution is defined by the density
\[ p_a(x) = \frac{2^{1-a/2} x^{a-1} e^{-x^2/2}}{\Gamma(a/2)}~. \]
We have almost surely:
\[ a_n^\beta = O(\sqrt{\log n})~, \, b_n^\beta = \sqrt{ n}(1 + o(1))~.  \]
Therefore by Propositions~\ref{limjac_1} and \ref{rhon}
\[ \Resc_{\sqrt { n}} \mu_n^\beta  \to \rho_\sc~, \quad
\Resc_{\sqrt { n}} \rho_n^\beta  \to \rho_\sc \quad \text{weakly,}\]
and by Corollaries~\ref{limjac_2} and \ref{limjac_3}
\[ \Resc_{\sqrt{ n}} \omega_n^\beta \to \Omega_\LSVK(x)~,\quad
\Resc_{\sqrt{ n}} \varpi_n^\beta  \to \Omega_\LSVK(x) \quad
\text{uniformly,}\]
 almost surely.

 \begin{rmk*} These statements are not directly related to the spectral properties
 of $J^\beta$, see \cite{BFS} and \cite{Br} for the properties of the latter and a
 discussion.
 \end{rmk*}

\subsection{Random matrices}

\subsubsection{}\label{guedef} The Gaussian Unitary Ensemble (GUE)
is the ensemble of semi-infinite Hermitian matrices $H = (H(i,j))_{i,j\geq1}$ such that for
any $n$ the top-left $n \times n$ submatrix $H_{n}$ has the
 probability density
\[ Z_{n}^{-1} \exp \left\{ - \frac12 \tr H_n^2 \right\}  \]
with respect to the Lebesgue measure on the space of $n \times n$
Hermitian matrices.

Let  $\lambda_j^{(n)}$ be the eigenvalues of $H_{n}$, and let $P_n(z) = \det(z-H_{n})$.
Define the spectral measure $\mu_n^{\mathrm{GUE}}$ and the normalised eigenvalue distribution $\rho_n^{\mathrm{GUE}}$
by the formul{\ae}
\[ \int x^k d\mu_n^{\mathrm{GUE}} = (H_n^k)_{nn}~, \quad \int x^k d\rho_n^{\mathrm{GUE}} = \frac{1}{n} \tr H_n^k~, \quad k=0,1,2,\cdots\]
Also define the diagrams $\omega_n^{\mathrm{GUE}}$ and $\varpi_n^{\mathrm{GUE}}$ by
\[ \frac{d}{dx} {\omega_n^{\mathrm{GUE}}}(x) = \operatorname{sign} \frac{P_{n-1}(x)}{P_n(x)}~, \quad
\frac{d}{dx} {\varpi_n^{\mathrm{GUE}}}(x) = \operatorname{sign} \frac{P_{n}'(x)}{P_n(x)}~. \]

\subsubsection{}\label{de2}

\begin{thm*}[Dumitriu--Edelman \cite{DE}] The joint distribution of 
\[ (\lambda_j^{(n)})_{j=1}^n \quad \text{and} \quad (\lambda_j^{(n-1)})_{j=1}^{n-1} \]
for the GUE coincides with the joint distribution of the same quantities for
 $J^{\beta=2}$ from \ref{ed}.
\end{thm*}

\noindent In other words, $J^{\beta=2}$ is obtained from the GUE by tridiagonalisation.

\begin{cor*} The distribution of $\mu_n^{\mathrm{GUE}},\rho_n^{\mathrm{GUE}},\omega_n^{\mathrm{GUE}},\varpi_n^{\mathrm{GUE}}$ from \ref{guedef}
coincides with the distribution of $\mu_n^{\beta=2},\rho_n^{\beta=2},\omega_n^{\beta=2},\varpi_n^{\beta=2}$ associated with the random
Jacobi matrix $J^{\beta=2}$ of \ref{ed}.
\end{cor*}

\subsubsection{}\label{kb1}

Combining Corollary~\ref{de2} with the conclusion of \ref{ed}, we obtain:
\begin{cor*}[Kerov \cite{Kerov_sep} /special case/] Almost surely
\[\begin{aligned}
 &\Resc_{\sqrt n} \mu_n^{\mathrm{GUE}},  \Resc_{\sqrt n} \rho_n^{\mathrm{GUE}}& &\to& \rho_\sc&\quad&\text{weakly}& \\
&\Resc_{\sqrt n} \omega_n^{\mathrm{GUE}},  \Resc_{\sqrt n} \varpi_n^{\mathrm{GUE}}& &\to& \Omega_\LSVK &\quad&\text{uniformly}&.\end{aligned}\]
\end{cor*}

\subsubsection{}\label{wig1} Corollary~\ref{kb1} can be extended to the class of Wigner matrices.
Let $H = (H(i, j))_{i,j\geq1}$ be an arbitrary semi-infinite Hermitian random matrix, such that
\[ \left\{ H(i, j)~, \, i \leq j \right\} \]
are jointly independent with $\mathbb{E} H(i,j)=0$, $(H(i, i))$ are identically distributed, and
$(H(i, j))_{i<j}$ are identically distributed with
$\mathbb{E}|H(i, j)|^2 = 1$. Let $H_n$ be the top-left principal submatrix of $H$,
and let $\mu_n, \rho_n, \omega_n, \varpi_n$ be defined as in \ref{guedef}.

Let us quote Wigner's law \cite{Wig}. In the current generality it was
proved by Pastur \cite{Pastur}.
\begin{prop*}[Wigner; Pastur] Almost surely $ \Resc_{\sqrt n} \rho_n \to \rho_{\sc}$.
\end{prop*}

\begin{cor*}[Kerov \cite{Kerov_sep}] Almost surely $ \Resc_{\sqrt n} \varpi_n \to \Omega_\LSVK$.
\end{cor*}

\subsubsection{}\label{kb2}

Proposition~\ref{wig1} and Corollary~\ref{wig1} have a counterpart for
$\mu_n$ and $\omega_n$. The former is the following version of
Wigner's law (proved similarly to \cite{Pastur}):
\begin{prop*} In the setting of \ref{wig1}, $ \Resc_{\sqrt n} \mu_n \to \rho_\sc$ almost surely.
\end{prop*}

\begin{cor*}[Kerov \cite{Kerov_sep}; Bufetov \cite{Buf} /general case/]
$\Resc_{\sqrt n} \omega_n \to \Omega_\LSVK$  almost surely.
\end{cor*}

\section{Corrections to the limit shape}\label{S:cor}

\subsection{Deterministic corrections}\label{s:detcor}

\subsubsection{}\label{fluct}

Let $\omega_n$ be a sequence of continual diagrams such that
$\omega_n \to \Omega$ in uniform topology. Then
\[ \mu_n = \Mar \omega_n \to \rho = \mathfrak M \Omega \]
in weak topology. Our goal is to relate the corrections
$\omega_n - \Omega$ to $\mu_n - \rho$. We start with a lemma.

\begin{lemma*} Let $\epsilon_n \to +0$. Suppose $\omega_n$ and $\Omega$
are continual diagrams such that the corresponding
measures $\mu_n = \Mar \omega_n$ and $\rho = \mathfrak M \Omega$ satisfy:
\begin{equation}\label{eq:fluct_assume} \int (x-z)^{-1} d\mu_n(x) = \int (x-z)^{-1} d\rho(x) + \epsilon_n R (z) + o(\epsilon_n)~,
\quad z \in \mathbb C \setminus \mathbb R~.\end{equation}
Then
\begin{equation} \int (x-z)^{-1} d(\omega_n(x)-\Omega(x))  =  -2 \epsilon_n \frac{R(z)}{w_\rho(z)} + o(\epsilon_n)~,\quad z \in \mathbb C \setminus \mathbb R~,\end{equation}
where we defined the Stieltjes transform
\begin{equation}
w_\rho(z) = \int (x-z)^{-1}d\rho(x)~.\end{equation}
\end{lemma*}

As one can see from the proof below,
if the convergence in the assumption is uniform on compact sets, it is so also
in the conclusion.

\begin{proof} By Theorem~\ref{hom} applied with $z^{-1}$ in place of $z$
\[\begin{split}
\int \log(1 - z^{-1}x) d\omega_n'(x)
&=  -2 \log \int(1-z^{-1}x)^{-1} d\mu_n(x)\\
&= - 2 \log z + 2 \log \int (x-z)^{-1} d\mu_n(x)~.
\end{split}\]
Using the assumption (\ref{eq:fluct_assume}) we deduce:
\[\begin{split}
&\int \log(1 - z^{-1} x) d\omega_n'(x)  \\
&=  - 2 \log z + 2 \log \int (x-z)^{-1} d\rho(x) + 2 \epsilon_n \frac{R(z)}{w_\rho(x)}
+ o(\epsilon_n)~.\end{split}\]
Similarly, $\Omega$ and $\rho$ are related by
\[ \int \log(1 - z^{-1} x) d\Omega'(x) = -2 \log z + 2 \log \int (x-z)^{-1} d\rho(x)~,\]
hence
\begin{equation}\label{eq:fluct_tmp} \int\log(1 - z^{-1} x) d(\omega_n'(x)-\Omega'(x)) = \frac{2 \epsilon_n R(z)}{w_\rho(x)} + o(\epsilon_n)~.
\end{equation}
Integrating by parts, we rewrite the left-hand side of (\ref{eq:fluct_tmp}) as
\[ \int\log(1 - z^{-1} x) d(\omega_n'(x)-\Omega'(x)) =
- \int (x-z)^{-1} d(\omega_n(x)-\Omega(x))\]
and this concludes the proof.
\end{proof}

\subsubsection{}\label{fluct2}

The conclusion of Lemma~\ref{fluct} can be reformulated in the following form,
from which one can see that the asymptotics of $\omega_n - \Omega$ (and not
just of the derivative) is determined.

\begin{cor*} In the setting of Lemma~\ref{fluct}, suppose that
the convergence  is uniform on compact subsets of $\mathbb C \setminus \mathbb R$.  Then
\begin{equation}\label{eq:fluct2} \int_{I} (x-z)^{-1} (\omega_n(x)-\Omega(x)) dx  = 2 \epsilon_n \int^z \frac{R(\zeta)d\zeta}{w_\rho(\zeta)} + o(\epsilon_n)~,\quad z \in \mathbb C \setminus \mathbb R~,\end{equation}
where $I$ is an interval in which $\Omega$ is supported, the integral is from $\pm i \infty$ along a path avoiding the real axis, and the convergence  is also uniform on compact subsets of $\mathbb C \setminus \mathbb R$.
\end{cor*}

\begin{proof}
Let $I$ be an interval in which $\Omega$ is supported. The assumption of uniform convergence on compact subsets implies that, for real $x$ outside $I$,
\begin{equation}\label{eq:suppconv}  \omega_n(x) - \Omega(x) = o(\epsilon_n)~.\end{equation}
Therefore we can choose $\delta_n = o(\epsilon_n)$ such that the shifted diagram
\[ \omega_n^\to(x) = \omega_n(x-\delta_n) \]
and $\Omega(x)$ have the same centre, i.e.\ coincide for sufficiently large $|x|$. Note that $\omega_n^\to(x) - \omega_n(x) = o(\epsilon_n)$ uniformly in $x$.
By Lemma~\ref{fluct},
\[ \int (x-z)^{-1} d(\omega_n^\to(x)-\Omega(x))  =  -2 \epsilon_n \frac{R(z)}{w_\rho(z)} + o(\epsilon_n)~,\quad z \in \mathbb C \setminus \mathbb R~,\]
whence, integrating by parts and  replacing $z$ with $\zeta$,
\begin{equation}\label{eq:wdiff} \frac{d}{d\zeta} \int (x-\zeta)^{-1} (\omega_n^\to(x)-\Omega(x))\, dx  =  2 \epsilon_n \frac{R(\zeta)}{w_\rho(\zeta)} + o(\epsilon_n)~,\quad \zeta \in \mathbb C \setminus \mathbb R~.\end{equation}
Integrating (\ref{eq:wdiff}), we obtain
\[ \int(x-z)^{-1} (\omega_n^\to(x)-\Omega(x)) dx  = 2 \epsilon_n \int^z \frac{R(\zeta)d\zeta}{w_\rho(\zeta)} + o(\epsilon_n)~,\quad z \in \mathbb C \setminus \mathbb R~,\]
where the integral can be taken over any interval containing the support of $\Omega$,
and this implies (\ref{eq:fluct2}).
\end{proof}

\subsubsection{}\label{fluct3}

Corollary~\ref{fluct2} allows to drag the corrections to the limiting shape through the
Markov correspondence. Denote by  $\mathfrak{B}_{[a,b]}$  the space of analytic test functions
$\phi: [a,b] \to \mathbb{C}$. The space $\mathfrak{B}_{[a,b]}$ is topologised 
as the projective limit of the spaces of analytic functions in shrinking neighbourhoods
of $[a, b]$. 
Also consider the space of continuous functionals $\mathfrak{B}_{[a, b]}'$, 
and observe that the topology on this space coincides with the minimal topology in which the functionals
\[ F \mapsto \langle F, (\bullet - z)^{-1} \rangle~, \quad z \notin [a, b] \]
are continuous. In the setting of \ref{fluct2}, define $F, F^\Mar \in \mathfrak{B}_{[a,b]}'$ by
\[ \langle F, \phi \rangle = \oint_\Gamma \phi(z) R(z) \frac{dz}{2\pi i}~, \quad
\langle F^\Mar, \phi \rangle = - \oint_\Gamma \Phi(z) \, \frac{R(z)}{w_\rho(z)}  \, \frac{dz}{\pi i}~, \]
where $\Phi(z) = \frac12 (\int_{-2}^z - \int_z^2) \phi(x) \, dx$, and $\Gamma$ encircles $I$ counterclockwise within the domain of analyticity of $\phi$.

\begin{prop*} Let $\mu_n, \rho$ be probability measures such that $\supp \rho \subset [a, b]$,
and let $\omega_n = \Mar^{-1} \mu_n$
and $\Omega = \Mar^{-1} \rho$. If, for some $\epsilon_n \to +0$,
\begin{equation}\label{eq:if} \epsilon_n^{-1} (d\mu_n(x) - d\rho(x)) \to F \quad \text{in $\mathfrak{B}_{[a, b]}'$}\end{equation}
then
\begin{equation}\label{eq:then} \epsilon_n^{-1} (\omega_n(x) dx - \Omega(x)dx ) \to F^\Mar \quad \text{in $\mathfrak{B}_{[a, b]}'$}\end{equation}
\end{prop*}

\begin{proof} Use Corollary~\ref{fluct2}  and the Cauchy theorem.
\end{proof}

By the construction of $\mathfrak B_{[-2,2]}'$, the left-hand side of (\ref{eq:then}) is
implicitly multiplied by the indicator $\mathbbm{1}_{[a,b]}(x)$. This indicator
can be dropped if
\[ \int x d\rho_n(x) = \int x d\rho(x) + o(\epsilon_n)~. \]
Otherwise, it is necessary, as observed in \cite{ES} and as one can see from Figure~\ref{fig}. Second, the implication in the proposition is in fact an equivalence: (\ref{eq:then}) implies (\ref{eq:if}), as one can see by tracing the arguments. 
\subsubsection{}\label{pi}

Consider the following example. Assume that $\rho = \rho_\sc$ is the semicircle
measure, the Stieltjes transform of which is given by
\[w_\rho(z) = \int_{-2}^2 \frac1{2\pi} \sqrt{4-x^2} \, \frac{dx}{x-z} = \frac{-z + \sqrt{z^2 - 4}}{2}~.\]
Recall the definition of Chebyshev polynomials
\begin{equation}\label{eq:chebdef}
T_n(\cos \theta) = \cos (n\theta)~, \quad U_n(\cos \theta) = \frac{\sin((n+1)\theta)}{\sin \theta}~, \end{equation}
or explicitly
\[\begin{aligned}
T_0(x/2) &=& 1~, \, T_1(x/2) &=& \frac{x}{2}~, \, T_2(x/2) &=& \frac{x^2}{2} - 1~, \, \cdots
\\
U_0(x/2) &=& 1~, \, U_1(x/2) &=& x~, \, U_2(x/2) &=& x^2 - 1~, \, \cdots
\end{aligned}\]
They satisfy the orthogonality relations
\[ \int T_k(x/2) T_l(x/2) d\rho_\arcsin(x) = \frac{1 + \delta_{k0}}{2} \delta_{kl}~, \quad
\int U_k(x/2) U_l(x/2) d\rho_\sc(x) = \delta_{kl}~.\]
Assume that
\begin{equation}\label{eq:assume}
\frac{1}{\epsilon_n} d(\mu_n(x) - \rho_\sc(x)) \to  \left\{ \sum_{k \geq 1} \frac{2 c_k  T_k(x/2) dx}{\pi \sqrt{4-x^2}} \right\} \quad \text{in $\mathfrak{B}_{[-2,2]}'$}
\end{equation}
i.e.\ for any $z \in \mathbb C \setminus [-2, 2]$,
\begin{equation}\label{fluct_ex1}
\int  \frac{d(\mu_n(x) - \rho_\sc(x))}{x-z}= \epsilon_n  \left\{ \sum_{k \geq 1} c_k \int_{-2}^2\frac{2  T_k(x/2) dx}{\pi \sqrt{4-x^2}(x-z) } \right\} + o(\epsilon_n)~.
\end{equation}
Observing that\[\begin{split}
w_{\rho_\arcsin,k}(z)&=
 \int_{-2}^2 \frac{ T_k(x/2) dx}{\pi \sqrt{4-x^2}(x-z) } =  \frac{-1}{\sqrt{z^2-4}}  T_k(z/2) + \frac12 U_{k-1}(z/2) \\
 &= -\frac1{\sqrt{z^2-4}} \left\{ \frac{z - \sqrt{z^2-4}}{2} \right\}^k~,
 \end{split}\]
 we deduce that for $k \geq 2$
 \[ -2 \frac{w_{\rho_\arcsin,k}(z)}{w_{\rho_\sc}(z)} = 2 w_{\rho_\arcsin,k-1}(z)
 = \frac{2}{k-1} \int_{-2}^2 \frac{\frac{d}{dx} \left[ U_{k-2}(x/2) \frac{\sqrt{4-x^2}}{2\pi} \right]}{x-z} dx~,\]
 whereas
 \[-2 \frac{w_{\rho_\arcsin,1}(z)}{w_{\rho_\sc}(z)} = 2 w_{\rho_\arcsin}(z)
 = - \frac{2}\pi \int_{-2}^2  \frac{\frac{d}{dx} \arcsin(x/2)}{x-z} dx~,\]
 and finally:
 \begin{equation}\label{eq:conclude}\begin{split}
 &\epsilon_n^{-1} (\omega_n(x) - \Omega_\LSVK(x))dx \longrightarrow \\
 &\quad \left\{ -  \frac{4 c_1}{\pi} \arcsin(x/2) + \sum_{k \geq 2} \frac{4 c_k}{k-1} U_{k-2}(x/2) \frac{\sqrt{4-x^2}}{2\pi} \right\} dx
 \quad\text{in $\mathfrak B_{[-2,2]}'$}\end{split}\end{equation}
We proved:
 \begin{prop*} Let $\{ c_k \}$ be a sequence with $\limsup |c_k|^{1/k} = 1$. If $\mu_n = \Mar \omega_n$ satisfy (\ref{eq:assume}), then (\ref{eq:conclude}) holds.
 \end{prop*}
 As we noted above, the arguments go in both directions, and  (\ref{eq:assume})
 is actually equivalent to (\ref{eq:conclude}).

\subsection{Fluctuations about the limit shape}

\subsubsection{}\label{clt}

Although Proposition~\ref{pi} was proved in deterministic setting, standard arguments
allow to apply it in  stochastic setting, i.e.\ for random measures $\mu_n$,
after proper modifications. For example, if the assumptions
holds almost surely, then so does the conclusion; if the assumption
holds in distribution, then so does the conclusion. The former version
follows directly from the deterministic statement, while the latter version follows
from the former one using Skorokhod's representation theorem. 
 In the sequel we work with convergence in distribution.

\subsubsection{}\label{joh}

Let $g_k$ be independent, identically distributed standard Gaussian random variables,
and let
\[\begin{split}
\Delta_{\tr}(x) &= \sum_{k \geq 1} \frac{\sqrt{k}}2 g_k \, \frac{2 T_k(x/2)}{\pi\sqrt{4-x^2}}~,  \\
\Delta_{\tr}^\Mar(x) &= - \frac{2}{\pi} g_1 \arcsin(x/2) - \sum_{k\geq 0} \frac{2 \sqrt{k+2}}{k+1} g_{k+2} \, \frac{U_{k}(x/2)\sqrt{4-x^2}}{2\pi}~,
\end{split}\]
where the series are understood in $\mathfrak B_{[-2,2]}'$.
Following \cite{IO}, we note the similarity between $\Delta_{\tr}(x)$ and
\[\frac{d}{dx} \Delta_{\tr}^\Mar(x) \sim - \frac{g_1}{\pi \sqrt{4-x^2}} - \sum_{k \geq 1}
\sqrt{k} g_{k+1} \frac{2 T_{k}(x/2)}{\pi\sqrt{4-x^2}}~.\]

\begin{thm*}[Johansson \cite{Joh}] For the Gaussian Unitary Ensemble,
\[n \, d(\Resc_{\sqrt{n}}[\rho_n^{\mathrm{GUE}}](x) - \rho_\sc(x)) \to \Delta_\tr(x) dx\] as random functionals on $\mathfrak{B}_{[-2,2]}$. \end{thm*}

In the original work \cite{Joh}, the convergence was established in slightly weaker
topology; now the result is available in the topology corresponding to $\mathfrak{B}_{[-2,2]}$ and even a much stronger one (cf.\ \ref{johtop} below). Appealing to (a stochastic version of)
Proposition~\ref{pi}, we obtain:

\begin{cor*} For the GUE,
\begin{align}
&n \, (\Resc_{\sqrt{n}}[\varpi_n^{\mathrm{GUE}}](x) - \Omega_\LSVK(x))dx \to \Delta_\tr^\Mar(x) dx\\
&  n \,  d \Resc_{\sqrt{n}} [\rho_n^{\mathrm{GUE}} -\rho_{n-1}^{\mathrm{GUE},*}](x)  \to \frac12 \frac{d^2}{dx^2} \Delta_\tr^\Mar(x)dx~,\label{eq:diffcrit}
\end{align}
where $\rho_{n-1}^*$ is the normalised counting measure of the zeros of $P_n'$.
\end{cor*}

\subsubsection{}\label{lp}

Now let
\[\begin{split}
\Delta_{1}(x) &= \sum_{k \geq 1} g_k U_k(x/2) \frac{\sqrt{4-x^2}}{2\pi}
= \sum_{k \geq 1} \frac{g_k}{2} \, \frac{2T_{k}(x) - 2T_{k+2}(x)}{\pi \sqrt{4-x^2}}\\
\Delta_{1}^\Mar(x) &= -\frac{2g_1}\pi \arcsin\frac{x}{2} - \sum_{k \geq1} \frac{2(g_{k}-g_{k+2})}{k+1} U_k(x/2)\frac{\sqrt{4-x^2}}{2\pi}~. \end{split}\]
We mention that $\Delta_1$ can also be described as follows:
\begin{equation}\label{eq:lp_ito}\Delta_1(x) dx=  \sqrt{\rho_\sc'(x)} dB(x) -\left\{ \int_{-2}^2 \sqrt{\rho_\sc'(y)} dB(y)  \right\} \rho_\sc'(x) dx~, \end{equation}
where $\rho_\sc'(x) = \frac{1}{2\pi} \sqrt{(4-x^2)_+}$, and $B(x)$ is the Brownian motion.
While $\Delta_1(x) dx$ is a generalised Gaussian process, its integral
\[ \Delta_1^\int (x) = \int_{-2}^{\min(x,2)}\Delta_1(y) dy \]
has a continuous modification, and so does $\Delta_{1}^\Mar(x)$.

\smallskip
The fluctuations of the measure $\mu_n$ are described by the following result
of Lytova--Pastur \cite{LP} (where a stronger topology defined by test
functions in $C^1$ is used, see further \ref{estop}).
\begin{thm*}[Lytova--Pastur \cite{LP} /weak form/] For the GUE,
\begin{equation}\label{eq:lp1} 
\sqrt{n} \, d(\Resc_{\sqrt{n}}[\mu_n^{\mathrm{GUE}}](x) - \rho_\sc(x)) \,\overset{\operatorname{distr}}\longrightarrow\, \Delta_{1}(x) dx\end{equation}
as random functionals on $\mathfrak B_{[-2,2]}$.
\end{thm*}
This theorem and Proposition~\ref{pi} imply:
\begin{cor*}[Erd\H{o}s and Schr\"{o}der \cite{ES} /special case, weak form/]
\begin{align}
\label{eq:lpes}
&\sqrt{n} (\Resc_{\sqrt{n}}[\omega_n^{\mathrm{GUE}}](x) - \Omega_\LSVK(x))dx 
\,\overset{\operatorname{distr}}\longrightarrow\, \Delta_{1}^\Mar(x) dx\\
\label{eq:diffsub}
&\sqrt n \, d \Resc_{\sqrt{n}} [\rho_n^{\mathrm{GUE}} -\rho_{n-1}^{\mathrm{GUE}}](x)  
\,\overset{\operatorname{distr}}\longrightarrow\, \frac12 \frac{d^2}{dx^2} \Delta_{1}^\Mar(x)dx~.
\end{align}
\end{cor*}

\subsection{On generalisations}

\subsubsection{}\label{joh_gen}

Johansson's Theorem~\ref{joh} has been extended beyond the Gaussian Unitary
Ensemble, see \cite[Section 18.4]{PShch}, \cite[Section 2.1.7]{AGZ} and references therein. We quote (\ref{eq:joh_gen}) below from the work of Bai and Yao \cite{BaiYao}; 
related results were proved earlier by Khorunzhy, Khoruzhenko, and Pastur \cite{KKP}.

 Let  $H$ be a Wigner matrix as in \ref{wig1}; 
let us assume that the matrix entries have finite fourth moments, and that 
$\mathbb{E} H(1,2)^2 = 0$ (the second condition can be omitted at the expense of
making the formul{\ae} more cumbersome). Then
\begin{equation}\label{eq:joh_gen}
n \, d(\Resc_{\sqrt{n}}[\rho_n^{\mathrm{GUE}}](x) - \rho_\sc(x)) \,\overset{\operatorname{distr}}{\longrightarrow} \,  \sum_{k \geq 1}\frac{a_k g_k + b_k}{2} \, \frac{2 T_k(x/2)}{\pi\sqrt{4-x^2}}
\end{equation}
as random functionals on $\mathfrak{B}_{[-2,2]}'$, where
\[
\begin{split}
&a_1 = \sqrt{\mathbb E H(1,1)^2}~, \,\, a_2 =  \sqrt{2\mathbb{E}(|H(1,2)|^2-1)^2}~, \,\, a_3 = \sqrt{3}~, \,\, 
a_4 = \sqrt{4}~, \,\, \cdots \\
& b_2 = \mathbb{E}H(1,1)^2 - 1~, \, \, b_4 = \mathbb{E}(|H(1,2)|^2-1)^2-1~,\,\,
b_1 = b_3 = b_5 = b_6 = \cdots = 0~.
 \end{split}\]
 Using (\ref{eq:joh_gen}) in place of Theorem~\ref{joh}, 
Corollary~\ref{joh} is extended as follows:
\begin{equation}\label{eq:joh_gen_cor}\begin{split}
&n \, (\Resc_{\sqrt{n}}[\varpi_n^{\mathrm{GUE}}](x) - \Omega_\LSVK(x))dx \,
\overset{\operatorname{distr}}\longrightarrow \, - \frac{2}{\pi} a_1 g_1 \arcsin(x/2) dx \\
&\qquad- \sum_{k\geq 0} \frac{2 \sqrt{k+2}}{k+1} (a_{k+2} g_{k+2} + b_{k+2}) \, \frac{U_{k}(x/2)\sqrt{4-x^2}}{2\pi}
 dx
\end{split}\end{equation}

\subsubsection{}\label{johtop}
It may also be possible
to strengthen the topology in Corollary~\ref{joh}, using as an input
a topologically stronger central limit theorem for linear statistics, such as the 
one proved by Shcherbina \cite{Shcherbina} (see further Sosoe and Wong \cite{SW}). 
To follow this strategy, one 
needs a version of Proposition~\ref{pi}  for non-analytic
test functions; this can be obtained, for example,  by
the method of pseudoanalytic extension \cite{Dyn1,Dyn2,HS}.

\subsubsection{}\label{lp2ors}

As for the fluctuations of the spectral measure, 
Theorem~\ref{lp} of Lytova--Pastur was extended beyond the Gaussian
ensembles by Lytova--Pastur \cite{LP2} and Pizzo--Renfrew--Soshnikov \cite{ORS}. We quote one of the results in \cite{ORS}:
for Wigner matrices as in \ref{wig1}
having finite fourth moments,
\begin{equation}\label{eq:lp2ors} \sqrt{n} \, d(\Resc_{\sqrt{n}}[\mu_n^{\mathrm{GUE}}](x) - \rho_\sc(x)) \,\overset{\operatorname{distr}}\longrightarrow \,\sum_{k \geq 1} h_k U_k(x/2) d\rho_\sc(x)~,\end{equation}
where $h_k$ are independent random variables: $h_1= - H(n,n)$, $h_2$ is centred Gaussian with variance depending on the moments of 
$H(1, 2)$, and $(h_k)_{k \geq 3}$ are centred Gaussian with variance depending 
on the symmetry class of $H$ (cf.\ (\ref{eq:esvar}) below for explicit formul{\ae}). The topology is stronger than that of $\mathfrak B_{[-2,2]}'$, see \ref{estop} below.

Applying Proposition~\ref{pi}, we deduce that, for $H$ as above,
\begin{equation}\label{eq:lp2orscor}\begin{split} 
&\sqrt{n} (\Resc_{\sqrt{n}}[\omega_n^{\mathrm{GUE}}](x) - \Omega_\LSVK(x))dx \\
&\qquad\overset{\operatorname{distr}}\longrightarrow \,
- \frac{2h_1}{\pi} \arcsin \frac{x}{2} \, dx - \sum_{k \geq 1} \frac{2(h_k - h_{k+1})}{k+1} U_k(x/2)d\rho_\sc(x)
\end{split}\end{equation}
in the topology of random functionals on $\mathfrak B_{[-2,2]}$. 

\subsubsection{}\label{estop}

The result of Erd\H{o}s and Schr\"oder \cite[Theorem 3]{ES} is topologically
much stronger than our statements in (\ref{eq:lpes}) and (\ref{eq:lp2orscor}), and applies, 
for example, to test functions which are indicators of intervals. This raises
the question whether the topology can be strengthened in the central
limit theorems (\ref{eq:lp1}), (\ref{eq:lp2ors}) for the spectral measure.

In the case of the GUE one can use the convergence to the semicircle on scale $n^{-1/2}$ and a functional central limit theorem for sums of independent exponentially distributed random
variables to obtain the following version of Theorem~\ref{lp}:
\begin{equation}\label{eq:upgrade}
\sqrt{n} \, \left(\mu_n^{\mathrm{GUE}}(x\sqrt{n} ) - \rho_\sc(x) \right)\overset{\operatorname{distr}}\longrightarrow \Delta_{1}^\int(x)
\end{equation}
as random continuous functions (here we identify measures with their cumulative
distribution functions). 

In the case of general Wigner matrices with finite fourth moments, the results of \cite{ORS} imply that (\ref{eq:lp2ors}) holds in the topology of functionals
on test functions $\phi \in C^7[-2-\delta,2+\delta]$.
It is also proved in \cite{ORS} that for Wigner
matrices the entries of which satisfy the Poincar\'e inequality the limit 
theorem (\ref{eq:lp2ors}) holds for Lipschitz test functions. This topology is still
(most probably) insufficient to recover the results of \cite{ES} in their full strength.

\subsubsection{}\label{esfluct}

In a remark in Section~\ref{intr:fl} (following (\ref{eq:lp})), we suggested that the results of
 \cite{ES}  could be pulled through the Markov correspondence
to obtain a version of  (\ref{eq:lp2ors}) in stronger topology. However,
Erd\H{o}s and Schr\"oder found \cite{ES2} a direct approach to the latter problem.

For the special case of Wigner matrices as in \ref{wig1}, their result asserts the
following (the more general setting of \cite{ES2} applies to matrices the  entries are not necessarily identically distributed): 
if $H(1,1)$ and
$H(1,2)$ have moments of arbitrary order, then
\begin{equation}\label{eq:esvar}\begin{split} 
&\sqrt{n}  (d\Resc_{\sqrt{n}} \mu_n^{H}(x) - d\rho_\sc(x)) \,
\overset{\operatorname{distr}}{\longrightarrow}\,
 d\rho_\sc(x) \Bigg[ - H(n,n) \, U_1(x/2) \\
 &+ \sqrt{\mathbb{E} |H(1,2)|^4 -1} \, g_2 U_2(x/2)  +
 \sum_{k \geq 3}  \sqrt{1 + (\mathbb{E} H(1,2)^2)^k} \, g_k U_k(x/2) \Bigg]~, \end{split}\end{equation}
 where  $g_k$ are independent standard Gaussian,
and the topology is defined by test functions $\phi: \mathbb R \to \mathbb R$ of bounded variation in $[-3,3]$.
For the convenience of the reader, we have not modified the  phrasing of the remark in Section~\ref{intr:fl}.

\section{Random partitions and random matrices}\label{S:part}

\subsection{Plancherel growth}

\subsubsection{}\label{young}

A partition $\lambda^n$ of $n \geq 0$ (denoted $\lambda^n \vdash n$)
is a non-increasing sequence
\[ \lambda^n = (\lambda^n_1 \geq \lambda^n_2 \geq \lambda^n_3 \geq \cdots)\]
of non-negative integers adding up to $n$. We identify a partition $\lambda^n$
with the Young diagram with rows of lengths $\lambda^n_1, \lambda^n_2, \cdots$,
i.e.\ the union of unit squares (boxes) with a corner at 
\[ \left\{(j, k) \, \mid \, 1 \leq j~, \, 1 \leq k \leq \lambda^n_j \right\}~. \]
For example, the partition $19 = 7 + 4 + 4 + 3 + 1$ of $19$ corresponds to
the Young diagram in Figure~\ref{fig:young} (left).
The content of a box $\square=(j,k)$ is by definition $\ct(\square) = k-j$. For two
diagrams $\lambda^n \vdash n$ and $\lambda^{n+1} \vdash n+1$ we write
$\lambda^n \nearrow \lambda^{n+1}$ if $\lambda^{n+1}$ is obtained by
adding a single box $\lambda^{n+1} \setminus \lambda^n$ to $\lambda^n$.
A box $\square\in \lambda^n$ is called an inner corner of $\lambda^n$
if $\lambda^n \setminus \square$ is a partition (Young diagram). A box $\square \notin \lambda^n$
is called an outer corner if $\lambda^n \cup \square$ is a partition. The continual diagram
$\omega[\lambda^n]$ associated to a partition $\lambda^n$ is defined by
\[\begin{split}
&\omega'(x) =-1
+ 2\, \# \left\{ \text{outer corners $\square$ with $\ct(\square) < x$} \right\} \\
&\qquad\qquad\quad-2 \, \#  \left\{ \text{inner corners  $\square$ with $\ct(\square) < x$} \right\} \\
&\omega(x) = |x| \quad \text{for sufficiently large $x$}
\end{split} \]
See Figure~\ref{fig:young} (right).

\subsubsection{}\label{plan}
A standard Young tableau is a chain
\begin{equation}\label{eq:chain}
\tab = (\lambda^0 = \varnothing \nearrow \lambda^1 \nearrow \lambda^2 \nearrow
\cdots \nearrow \lambda^n)~.\end{equation}
The dimension $\dim \widetilde\lambda^n$ of a Young diagram $\widetilde\lambda^n$ is the number of chains (\ref{eq:chain}) with $\lambda^n = \widetilde\lambda^n$. In the representation theory of
the symmetric group, the irreducible representations are in one-to-one correspondence
with the partitions of $n$, and $\dim \lambda$ equals the dimension of the irreducible
representation corresponding to $\lambda$. One has (see \ref{repr} below):
\begin{equation}\label{eq:dims} \dim \lambda^n = \sum_{\lambda^{n-1} \nearrow \lambda^n} \dim \lambda^{n-1}~, \quad
\sum_{\lambda^n \vdash n} \dim^2 \lambda^n = n! \end{equation}

\subsubsection{}\label{tran}
Consider the space $\Tab_\infty$ of infinite sequences
\begin{equation}\lambda^0 = \varnothing \nearrow \lambda^1 \nearrow \lambda^2 \nearrow
\cdots \nearrow \lambda^n \nearrow \cdots
\end{equation}
equipped with product topology. Define a probability distribution on $\Tab_\infty$
by
\[ \mathbb{P} \left\{ \lambda^1 = \widetilde\lambda^1~, \cdots,
\lambda^n = \widetilde\lambda^n \right\} = \frac{\dim \widetilde\lambda^n}{n!} \]
for any $n$ and any chain $\widetilde\lambda^0\nearrow\cdots\nearrow\widetilde\lambda^n$. The corresponding process is called Plancherel growth, and the distribution
\[ \mathbb{P}\left\{
\lambda^n = \widetilde\lambda^n \right\} = \frac{\dim^2 \widetilde\lambda^n}{n!} \]
of $\lambda^n$ --- the Plancherel measure.
The transition measure of a diagram $\widetilde\lambda^n \vdash n$ is the
probability measure
\begin{equation}\label{eq:trm}
\mu[\widetilde\lambda^n] = \sum_{\text{outer corners $\square$}}
\frac{\mathbb{P}\left\{ \lambda^{n+1} = \widetilde\lambda^n \cup \square~, \, \lambda^n = \widetilde\lambda^n\right\} }{\mathbb{P}\left\{ \lambda^n = \widetilde\lambda^n\right\} } \,\, \delta_{\ct(\square)}
\end{equation}
This measure was introduced by Kerov \cite{Kerov_tr}, who proved  that $\Mar \omega[\lambda^n] = \mu[\lambda^n]$.

\begin{thm*}[Logan--Shepp \cite{LS}, Vershik--Kerov\cite{VK1,VK2}; Kerov \cite{Kerov_sep}]
\[ \Resc_{\sqrt{n}} \left\{ \omega[\lambda^n] \right\}\to \Omega_\LSVK \quad \text{and} \quad \Resc_{\sqrt{n}} \mu[\lambda^n] \to \rho_\sc \]
(in distribution and almost surely).
\end{thm*}
The first part was proved by Logan--Shepp \cite{LS} and Vershik--Kerov \cite{VK1,VK2},
while the second part was deduced by Kerov \cite{Kerov_sep} using the Markov
correspondence. Vice versa,
one can first prove the second part (see Biane \cite{Bi} and
\ref{kermean}, \ref{meantodistr} below), and then deduce the first part.

\subsubsection{}\label{repr}
The construction above arises in the representation theory of the symmetric group.
Let us explain the minimum that we need below,
and refer to \cite{Kerov_book} for more details.
The left regular representation of the symmetric group $S_n$ is the space
$\mathbb{C}[S_n]$ of
linear combinations $ \sum_{\pi \in S_n} a_\pi \pi$, equipped with the action
of $S_n$ by left multiplication. It can be decomposed into a sum of
irreducible representations, so that the multiplicity of the representation
corresponding to $\lambda^n$ is equal to $\dim \lambda^n$. This implies the second equality
in (\ref{eq:dims}). 

The decomposition of an irreducible representation corresponding to 
$\lambda^n \vdash n$ into irreducible representations of $S_{n-1}$ contains
exactly the representations corresponding to $\lambda^{n-1} \nearrow \lambda^n$,
and the multiplicity of each of these is exactly one. This implies the first
equality in (\ref{eq:dims}). Furthermore, it follows that an irreducible representation corresponding to $\lambda^n\vdash n$ is decomposed
into one-dimensional subspaces corresponding to chains $\tab$ of the
form (\ref{eq:chain}).

The arguments below also rely on the properties of the Jucys--Murphy elements $X_m \in \mathbb{C}[S_n]$, which are defined as
sums of transpositions
\[ X_m = (1 \, m) + (2 \, m) + \cdots + (m-1 \, m)~. \]
Every $\tab$ is invariant under all $X_m$, and $X_m\mid_\tab = \ct(\lambda^m \setminus \lambda^{m-1})$ (see \cite{OV}, where the Jucys--Murphys elements are
used to reconstruct the representation theory of the symmetric group,
and \cite{Bi,Ok}). This implies
\[ \frac{1}{n!} \tr P(X_{n}) = \mathbb{E} \int P \, d\mu[\lambda^{n-1}] \]
for any polynomial $P$.  The use of Jucys--Murphy elements to compute moments of
the transition measure goes back to the work of Biane \cite{Bi}.

\subsubsection{}\label{kerov_clt}

Define two Gaussian processes
\begin{align}\label{eq:gkerov}
\Delta_{\mathrm{part}}(x) &= \sum_{k \geq 3} \frac{\sqrt{k-1}}{2} g_{k} \, \frac{2 T_k(x/2)}{\pi \sqrt{4-x^2}}\\
(-)\Delta_{\mathrm{part}}^\Mar(x) &= \sum_{k \geq 1} \frac{2 g_{k+2}}{\sqrt{k+1}} \, \frac{U_{k}(x/2) \sqrt{4-x^2}}{2\pi}
\end{align}

\begin{thm*}[Kerov, Ivanov--Olshanski \cite{Kerov_CLT,IO}]
\[ \sqrt{n} \left\{ \Resc_{\sqrt{n}} \omega[\lambda^n] -  \Omega_\LSVK \right\} dx \to \Delta_{\mathrm{part}}^\Mar(x) dx\]
and
\[ \sqrt{n} \left\{ d\Resc_{\sqrt{n}} \mu[\lambda^n](x) -  d\rho_\sc(x) \right\} \to \Delta_{\mathrm{part}}(x) dx\]
in distribution, as random functionals on $\mathfrak B_{[-2,2]}$.
\end{thm*}

The first part was proved by Kerov \cite{Kerov_CLT} in 1993. A simplified proof, based on Kerov's notes, was published by
Ivanov and Olshanski \cite{IO}. The second part was proved in \cite{IO},
independently of Theorem~\ref{kerov_clt}
(although by a similar method). In view of Proposition~\ref{clt}, the second
part implies the first part (and vice versa). In Section~\ref{s:pfs}, we  sketch a
proof of the second part,
using a combinatorial approach which was used by Biane \cite{Bi} and by Okounkov \cite{Ok}, and recently developed in \cite{JS}
As pointed out in \cite{IO}, $\Delta_{\mathrm{part}}$ is similar to $\Delta_{\mathrm{tr}}$
from \ref{joh}. It is even more similar to the process from Proposition~\ref{unimod}.
 The argument
 will highlight this similarity, and also explain the appearance of the factor
$\sqrt{k-1}$ in (\ref{eq:gkerov}) in place of $\sqrt k$ in (\ref{eq:joh_clt}).

\subsubsection{}\label{unimod}

To emphasise the similarity between random matrices and random partitions, we
consider the following special Wigner matrix:
$H = H^{\textrm{unif}} = (H(i, j))$ with
\[ H(i, i) = 0~, \quad H(i, j) \sim \operatorname{Unif}(S^1) \quad (i < j)~. \]
This ensemble is particularly convenient due to the identities (\ref{eq:chebviawalk}),
see \cite{me_surv} for a survey of applications of relations of this kind, and
historical remarks. (For other ensembles of Wigner matrices (\ref{eq:chebviawalk})
is no longer an identity, however, the difference between the left-hand side
and the right-hand side is a small quantity, cf.\ \cite{FeldS}.)
For the ensemble $H^{\textrm{unif}}$, the central limit theorem is stated as follows.

\begin{prop*}
For $H^{\textrm{unif}}$,
\[\begin{split} &n \, d(\Resc_{\sqrt{n}}[\rho_n^{\mathrm{unif}}](x) - \rho_\sc(x)) \to \sum_{k \geq 3} \frac{\sqrt{k}}{2} g_k \frac{2T_k(x/2)dx}{\pi \sqrt{4-x^2}}\\
&n \, (\Resc_{\sqrt{n}}[\varpi_n^{\mathrm{unif}}](x) - \Omega_\LSVK(x))dx \to \sum_{k \geq 1} \frac{2\sqrt{k+2}}{k+1} g_{k+2} \frac{U_k(x/2)\sqrt{4-x^2}\, dx}{2\pi}~.
\end{split}\]
\end{prop*}

\noindent In the next section, we prove this proposition
in parallel with Theorem~\ref{kerov_clt}.

\subsection{Asymptotics of moments}\label{s:pfs}

\subsubsection{}

Consider the following sequence of polynomials:
\begin{equation}\label{eq:viacheb}
	P_{l,m}(x) = (m-1)^{\frac{l}{2}}  U_l\left( \frac{x}{2\sqrt{m-1}} \right) - (m-1)^{\frac{l-2}{2}} U_{l-2}\left( \frac{x}{2\sqrt{m-1}} \right)  
	\end{equation}
with the convention $U_{-2} \equiv U_{-1} \equiv 0$. For  $\mathcal p = (u_0, u_1,
\cdots, u_l)$, set
\[\widehat \pi (\mathcal p; H) = H(u_0, u_1) H(u_1, u_2) \cdots H(u_{l-1}, u_l)~. \]
Also let
\[ \widehat{\mathcal P}_{l,n}(u, v) = \left\{ (u_0, u_1, \cdots, u_{l-1}, u_l) \in \{1,\cdots,n\}^l\,\Big|\,
u_0 = u~, \, u_l = v~, \, u_r \neq u_{r-1}, u_{r-2} \right\}~. \]
Graphically,
we may represent a tuple in $\widehat{\mathcal P}_{l,n}(u, v)$ as a path of length
$l$ from $u$ to $v$ which does not backtrack.
\begin{lemma*}[e.g. {\cite{me_surv}}]
For any Hermitian $H$ with $|H(u,v)| = 1 - \delta_{uv}$ and any $u, v$,
\begin{equation}\label{eq:chebviawalk}
P_{l,n-1}(H_n)(u, v) = \sum_{\mathcal p \in  \widehat{\mathcal P}_{l,n}(u, v)} \widehat\pi(\mathcal p; H)~,
\end{equation}
and consequently
\[ \mathbb E \prod_r P_{l_r,n-1}(H)(u_r, v_r) =
 \left\{ \underline{\mathcal{p}}  \in \prod_r \widehat{\mathcal P}_{l_r, n} \, \Big| \,
\prod_r \widehat\pi(\mathcal p_r; H)  \equiv 1 \right\}~. \]
\end{lemma*}
Graphically, the product is identically one if every edge is traversed forward
the same number of times as backward.

\subsubsection{}\label{mainjmrel}
Now we state (following \cite{JS}) the counterpart of (\ref{eq:viacheb}) for random partitions.
Let $\mathcal p = (j_1, j_2, \cdots, j_l)$ be a sequence of numbers. Denote:
\[ \pi_m(\mathcal p) = (j_1 \, m) (j_2 \, m) \cdots (j_l \, m)~.\]
Also let
\begin{equation}\label{eq:irred}
\mathcal P_{l,m} = \left\{ \mathcal p \in \{1, \cdots, m-1\}^l \, \mid \, \forall 1 \leq r \leq l-1\, j_r \neq j_{r+1}\right\} ~.
\end{equation}
The following lemma
can be checked by induction:
\begin{lemma*}[{\cite[Lemma 4.1]{JS}}]
\[ P_{l,m-1}(X_m) = \sum_{\mathcal p \in \mathcal P_{l,m}} \pi_m(\mathcal p)~, \]
and consequently
\begin{equation}  \mathbb \tr  \prod_{r} P_{l_r,m_r-1}(X_{m_r}) =
\# \left\{ \underline{\mathcal{p}}  \in \prod_r \mathcal P_{l_r, m_r} \, \Big| \,
\prod_r\pi_{m_r}(\mathcal p_r)   = 1 (= \operatorname{id}_{S_n}) \right\}~.
\end{equation}
\end{lemma*}

\subsubsection{}\label{wigmean}
Let us prove that Wigner's law holds in the mean:
\begin{equation}\label{eq:wigmean}
 \mathbb{E} \Resc_{\sqrt n} \rho_n^{\textrm{unif}} = 
 \mathbb{E} \Resc_{\sqrt n} \mu_n^{\textrm{unif}} \to \rho_\sc~.
 \end{equation}
We omit many details, which can be found, for example, in \cite[2.4.1]{me_surv}.

\begin{proof}[Proof of Wigner's law in the mean]
The  polynomials $U_k(x/2)$ are orthogonal
with respect to $\rho_\sc$, therefore to prove (\ref{eq:wigmean})
it suffices to show that
\[ \lim_{n \to \infty}  \mathbb{E} \int U_k(x/2)d \Resc_{\sqrt n} \rho_n^{\textrm{unif}}(x) =
\lim_{n \to \infty}  \mathbb{E} \int U_k(x/2)d \Resc_{\sqrt n} \mu_n^{\textrm{unif}}(x) = 0~, \,\, k = 1,2,\cdots,\]
which, by (\ref{eq:viacheb}), is equivalent to
\begin{equation}\label{eq:wigmeantag}
\lim_{n \to \infty} n^{-\frac{k}2} \mathbb{E}  P_{k,n-1}(H)(n,n) = 0~, \quad k \geq 1~,
\end{equation}
and we do this  using Wigner's power-counting argument, as follows.

The quantity
 $\mathbb{E} P_{k,n}(H)(n,n)$ counts the number of paths in $\mathcal{P}_{k,n-1}(n,n)$ which satisfy the parity condition
\[ \forall a \neq b \quad \# \left\{ j~, \, u_j = a~, \, u_{j+1} = b \right\} =
 \# \left\{ j~, \, u_j = b~, \, u_{j+1} = a \right\} ~, \]
where we set $u_0 = u_k = n$. First, we divide $\mathcal{P}_{k,n-1}(n,n)$  into isomorphism classes: $\mathcal p$
is isomorphic to $\mathcal p'$ if $\mathcal p = \sigma \circ \mathcal p'$ for some
permutation $\sigma \in S_n$. The number of isomorphism classes is bounded
as $n \to \infty$, for every $k$. Second, to every class we associate the characteristic
$\chi = V - E + 1$, where $V$ is the number of vertices and $E$ is the number of edges, and observe that $E \leq k/2$ and hence the contribution of a class is bounded  by $C_k n^{\frac{k}{2}+\chi-1}$. Finally, a path can not be tree-like (such as $(u_0, \cdots, u_{10}) = (j_1,j_2,j_3,j_4,j_3,j_5,j_6,j_5,j_3,j_2,j_1)$), since trees have leaves at which the path backtracks (in the example in the previous parentheses, $u_2 = u_4 = j_3$, et cet.), hence $\chi \leq 1$. A somewhat more careful argument shows that
$\chi \leq 0$ with equality for paths isomorphic to the one on \cite[Figure 4, right]{JS}, \cite[Figure 3.1, middle]{me_surv},
but this is not needed at the moment. Finally,
\[0 \leq n^{-\frac{k}2} \mathbb{E}  P_{k,n-1}(H)(n,n) \leq C_k' n^{-1}~.\]
This proves (\ref{eq:wigmeantag}) and (\ref{eq:wigmean}).
\end{proof}

\subsubsection{}\label{kermean} In the same way,  \ref{mainjmrel} can be used to prove that Theorem~\ref{tran} holds in the mean:
\[ \Resc_{\sqrt{n}} \mathbb E  \mu[\lambda^n] \to \rho_\sc~. \]
\begin{proof}[Proof of Kerov's law in the mean]
It suffices to show that
\begin{equation}\label{eq:tmp_comb1}
\lim_{n \to \infty} \frac{1}{n^{k/2}} \frac1{n!} \tr P_{k,n-1}(X_n) = 0~, \quad k =1,2,3,\cdots
\end{equation}
The argument, essentially due Biane (see \cite{Bi}), is similar to \ref{wigmean}.
We write
\[ \tr P_{k,n-1}(X_n) = \# \left\{ \mathcal p \in \mathcal P_{k,n} \, \mid \, \pi_n(\mathcal p ) = 1\right\} \]
and divide  the solutions to  the equation $\pi_n(\mathcal p) = 1$
into isomorphism classes (by $S_{n-1}$-conjugation).
Then we assign to every class the characteristic
\[ \chi = \# \{ \text{distinct indices in $\mathcal p$}\} - \frac{k}{2} +2~.\]
See \cite[4.2.1]{JS} for a graphical interpretation. The irreducibility condition
(\ref{eq:irred}) rules out tree-like solutions such as
\[ (j_1 \, n) \, (j_2 \, n) \, (j_3\, n) \, (j_3 \, n) \, (j_2 \, n) \, (j_4 \, n) \, (j_4 \, n)
\, (j_1 , n) = 1 \]
with $\chi = 2$, therefore $\chi \leq 1$ (in fact, $\chi \leq 0$, but this is not
required at the moment), and this implies (\ref{eq:tmp_comb1}).
\end{proof}

\subsubsection{}\label{meantodistr}
To prove that Wigner's law and Kerov's law hold in distribution, we need to show that also
the variance of
\[ \int U_k(x/2) d\mathcal{R}_{\sqrt n}\rho_n^{\textrm{unif}} \quad \text{and} \quad
\int U_k(x/2) d\mathcal{R}_{\sqrt n}\mu[\lambda^n] \]
tends to zero. This is accomplished via a  combinatorial argument which is similar to
that in \ref{wigmean} and \ref{kermean}: one shows that, for any isomorphism
class of pairs contributing to the variance, the number of distinct vertices / indices is
at most $k/2$. We omit the details.

\subsubsection{}\label{hwwig}
\noindent In order to consider  fluctuations, 
we introduce another family of polynomials 
\begin{equation}\label{eq:viacheb'}
	Q_{k,m}(x) =  2(m-1)^{\frac{k}{2}} \, T_k\left( \frac{x}{2\sqrt{m-1}} \right) - 2(m-1)^{\frac{k-2}{2}} \, T_{k-2}\left( \frac{x}{2\sqrt{m-1}} \right)
	\end{equation}
Then we have, for $k \geq 3$:
\begin{equation}\label{eq:qkcomb} \tr Q_{k,n-1}(H) = \sum_{\mathcal p \in\widehat{\mathcal P}_{k,n}^\circ}
\widehat{\pi}(\mathcal p, H)-\sum_{\mathcal p \in\widehat{\mathcal P}_{k-2,n}^\circ}
\widehat{\pi}(\mathcal p, H)~,\end{equation}
where
\[ \widehat{\mathcal P}_{k,n}^\circ= \left\{
\mathcal p \in \widehat{\mathcal P}_{k,n} \, \big| \, u_{k-1} \neq u_1\right\}\]
is the collection of cyclically non-backtracking paths.
The relation (\ref{eq:viacheb'}) is derived from (\ref{eq:viacheb}) using the identity
\[ U_k - U_{k-2} = 2 T_k~. \]
For $k=1,2$, we have:
\[ \tr 2 \sqrt{n-2} \, T_1\left(\frac{H}{2\sqrt{n-2}}\right) = 0~, \quad
\tr 2 (n-2) \,T_2\left(\frac{H}{2\sqrt{n-2}}\right) = -n(n-3)~.\]
Consequently, we obtain that
\begin{equation}\label{eq:tkcomb}
\tr 2  (n-2)^{\frac{k}{2}} T_k\left(\frac{H}{2\sqrt{n-2}}\right) = \sum_{\mathcal p \in\widehat{\mathcal P}_{k,n}^\circ}
\widehat{\pi}(\mathcal p, H) - n(n-3)\mathbbm{1}_{\text{$n$ is even}}~. 
\end{equation}
The relation (\ref{eq:tkcomb}) was first derived by Oren, Godel
and Smilansky \cite{OGS}, using a trace formula (see further \cite{OS1,OS2} and
\ref{trace} below).
Related combinatorial interpretations of the traces of Chebyshev polynomials
of the first kind appeared in the works of Anderson and Zeitouni \cite{AZ} (who
used generating functions),
and of Kusalik, Mingo and Speicher \cite{KMS} and Schenker and Schulz-Baldes \cite{SchSch}.

\begin{proof}[Proof of Proposition~\ref{unimod}]
The proposition is derived from (\ref{eq:tkcomb}) using an approximation
argument and 
the following two claims.
\begin{cl} Let $\widehat{\mathcal C}_{k,n}$ be the sub-collection of paths in $\widehat{\mathcal{P}}^\circ_{k,n}$   in which all the vertices are distinct. Then
\begin{equation}\label{eq:clA}
(n-2)^{-\frac{k}{2}}   \sum_{\mathcal p \in\widehat{\mathcal P}_{k,n}^\circ \setminus \widehat{\mathcal C}_{k,n}} \widehat{\pi}(\mathcal p, H) \to 0\end{equation}
in distribution.
\end{cl}
\begin{cl} Let $g_3, g_4, \cdots$ be independent standard Gaussian variables. Then
\begin{equation}\label{eq:Skn}
\left\{ (n-2)^{-\frac{k}{2}} \sum_{\mathcal p \in \widehat{\mathcal C}_{k,n}} \widehat{\pi}(\mathcal p, H)\right\}_{k \geq 3} \to \left\{  \sqrt{k} \, g_k \right\}_{k \geq 3}
\end{equation}
in distribution.
\end{cl}
Claim~A is proved by power-counting as in the \ref{wigmean} above: the elements of
$\widehat{\mathcal C}_{k,n}$ have $V = E$, whereas for all the other
paths in $\widehat{\mathcal{P}}^\circ_{k,n}$ $V < E$. This implies that
$\mathbb E [\text{LHS of} \, (\ref{eq:clA})]^2 \to 0$.

To prove Claim~B, denote the expression inside the braces on
the left-hand side of (\ref{eq:Skn}) by $S_{k,n}$; then (and this is the main
combinatorial step in the proof)
\[\begin{split} & \lim_{n\to\infty}  \mathbb{E} \prod_{j=1}^r S_{k_j,n} \\
&=
\sum_{\text{pairings of $\{1,\cdots,r\}$}} \prod_{\text{pair $(j, j')$}}
\# \left\{ \text{alignments of a $k_j$-cycle with a $k_{j'}$-cycle}\right\}\\
&= \sum_{\text{pairings of $\{1,\cdots,r\}$}} \prod_{\text{pair $(j, j')$}}
k_{j} \, \delta_{k_j,k_{j'}}~,\end{split} \]
where the first equality is a consequence of power-counting. Claim~B
follows by the Wick rule.\qedhere
\end{proof}

\subsubsection{}\label{hwpart}
\begin{proof}[Proof of Proposition~\ref{kerov_clt}]
Let us denote by $\operatorname{move}_m$ any linear combination
of permutations for which $m$ is not a fixed point. Then (for $k \geq 3$)
\[ Q_{k,m-1}(X_m) = \sum_{\mathcal p \in{\mathcal P}_{k,m}^\circ}
{\pi}_m(\mathcal p)  + \operatorname{move}_m~,\]
where
\[ {\mathcal P}_{k,m}^\circ = \left\{
\mathcal p \in {\mathcal P}_{k,m} \, \big| \, u_{k} = u_1~, u_{k-1} \neq u_2 \right\}\]
is the collection of cyclically non-backtracking paths. In particular, 
\[ \tr \prod_{j=1}^r Q_{k_j,m_j-1}(X_{m_j}) = \tr \prod_{j=1}^r
\sum_{\mathcal p \in{\mathcal P}_{k_j,m_j}^\circ}
{\pi}_{m_j}(\mathcal p)~,\]
if $m_1 < \cdots < m_r$.
Consider the sub-sub-collection
\[ \mathcal{C}_{k,m} = \left\{ (j_1, \cdots, j_k < m \, \mid \,  \text{$j_k = j_1$ and $j_1, \cdots, j_{k-1}$ are pairwise distinct}\right\}~.\]
(Note that the cycles have
$k-1$ vertices, as opposed to the $k$-long cycles of \ref{hwwig}.)
Then Claim \ref{hwwig}.A  takes the form
\[ n^{-\frac{k}{2}}  \left[ Q_{k,n-1}(X_n)|_{\tab} - \sum_{\mathcal p \in {\mathcal C}_{k,n}} {\pi}_n(\mathcal p)|_{\tab} \right] \longrightarrow 0~,\]
where $\tab$ is sampled from the Plancherel growth process, 
whereas in place of Claim \ref{hwwig}.B one has
\[ \left\{ \sum_{\mathcal p \in {\mathcal C}_{k,n}} {\pi}_n(\mathcal p)|_{\tab}
\right\}_{k \geq 3} \longrightarrow \left\{  \sqrt{k-1} \, g_k \right\}_{k \geq 3}~.\]
Both claims are proved by evaluating moments, where we take
$m_j = n-j+1$.  Then Proposition~\ref{kerov_clt} follows from the claims
by an approximation argument.

\end{proof}

\subsection{Some comments}

\subsubsection{}The counterpart of Theorem~\ref{lp} (i.e.\ the special
case of (\ref{eq:lp2ors})) for the ensemble
$H^{\mathrm{unif}}$ is stated as follows:
\begin{equation}\label{eq:clt_loc_unif}
\sqrt{n} \, d(\Resc_{\sqrt{n}}[\mu_n^{\mathrm{unif}}](x) - \rho_\sc(x)) \,\overset{\operatorname{distr}}\longrightarrow\,
\sum_{k \geq 3} g_k U_k(x/2) \frac{\sqrt{4-x^2}}{2\pi} dx~,\end{equation}
while (\ref{eq:lp2orscor}) takes the form
\begin{equation*}
\sqrt{n} (\Resc_{\sqrt{n}}[\omega_n^{\mathrm{unif}}](x) - \Omega_\LSVK(x))dx \to
\sum_{k \geq 3} 2g_k \left\{ \frac{U_{k-2}(x/2)}{k-1} - \frac{U_{k}(x/2)}{k+1} \right\} \frac{\sqrt{4-x^2}}{2\pi} dx~. \end{equation*}
The  coefficient ``$1$'' in front of $g_k$ in (\ref{eq:clt_loc_unif}) is combinatorially
interpreted as the number of ways to align two cycles of length $k$ with a marked
vertex.

\subsubsection{}\label{transp}
The parallelism between random matrices and random Young diagrams
is somewhat more transparent if described via the  matrix of transpositions
\[ \Gamma_n = \left( \begin{array}{ccccc}
0 & (12) & (13) & \cdots & (1n)\\
(12) & 0 & (23 ) & \cdots & (2n)\\
&&\ddots&&\\
(1n) & (2n) & (3n ) & \cdots & 0\\
\end{array} \right) \]
introduced (in a different gauge) by Biane \cite{Bi}. The spectral
properties of the restriction $\Gamma_n|_{\lambda^n}$
to $\lambda^n  \otimes \mathbb C^n$,
where $\lambda^n$ is an irreducible representation, are
similar to those of the Jucys--Murphy elements; if $\lambda^n$ is
randomly sampled from the Plancherel measure, the corresponding
matrix $\Gamma_n|_{\lambda^n}$ can be considered
as a counterpart to the random matrix $H_n$.

In particular, the use of $\Gamma_n$ in place of $X_n$ makes
the argument in \ref{hwpart} even more similar to that in \ref{hwwig},
and also introduces a conceptual simplification by allowing
to work with central elements only.

\subsubsection{}\label{trace}
The combinatorial constructions of Section~\ref{S:part} can be recast
in terms of Ihara-type zeta-functions \cite{Ihara,Bass}.
For a Hermitian matrix $H$
such that $|H(u,v)|=1-\delta_{uv}$ (or, more suggestively, $H(u,v)H(v,u) = 1-\delta_{uv}$), consider the zeta functions
\[ \zeta_{H_n}(u) = \prod_{l\geq 0} \prod_{\mathcal p \in \widehat{\mathcal P}_{l,n}^\circ} (1 -  \widehat\pi(\mathcal p; H) u^l)^{-1}~.\]
A variant of the Bass determinantal formula \cite{Bass} asserts that
\begin{equation}\label{eq:bass1}
\zeta_{H_n}(u)^{-1} = (1-u^2)^{\binom{n-1}{2}-1} \det(\mathbbm 1 - u H_n + (n-1)u^2 \mathbbm 1)~.  \end{equation}
Similarly, consider the $\mathbb C[S_n]$-valued function
\[ \zeta_n(u) = \prod_{l\geq 0} \prod_{\mathcal p \in \mathcal P_{l,n}^\circ} (1 - \pi_n(\mathcal p) u^l)^{-1}~,\]
for which one has 
\[ \zeta_n(u)^{-1} = (1-u^2)^{\binom{n-1}{2}-1} 
\det\nolimits_{n \times n}(\mathbbm 1 - u \, \Gamma_n  + (n-1) \, u^2 \mathbbm 1)~.\]  
Both sides of this identity are central, and hence act as scalars on each irreducible representation $\lambda^n$ of $S_n$; thus we obtain:
\begin{equation}\label{eq:bass2}
\left[\zeta_n(u)|_{\lambda^n} \right]^{-1} = (1-u^2)^{\binom{n-1}{2}-1} 
\det\nolimits_{n \times n}(\mathbbm 1 - u \, \Gamma_n|_{\lambda^n} + (n-1) \, u^2 \mathbbm 1)~.\end{equation}

\smallskip\noindent
The logarithms of the relations (\ref{eq:bass1}) and (\ref{eq:bass2}) are trace formul{\ae} relating the spectra of $H_n$
and $\Gamma_n$ with the quantities $\widehat \pi_n$ and
$\pi_n$, respectively.

The trace formul{\ae} thus obtained do not explicitly separate between the semicircle
and the corrections to it. Improved trace formul{\ae}, in which such a separation
is explicit, were obtained, in the setting of $d$-regular graphs,
by Oren, Godel and Smilansky, see  \cite{OGS,OS1} and further \cite{OS2}.
Their approach  can be applied to the problems discussed here.

\subsubsection{}
Chebyshev polynomials appear naturally in the combinatorial approach described
in Section~\ref{s:pfs}, and also in the trace formul{\ae} mentioned in \ref{trace}. Another approach in which their r\^ole is apparent was
developed by Joyner and Smilansky \cite{JS1,JS2}; it relies on the study
of the Fokker-Planck equation describing the evolution of the ensemble under
random walk. We refer in particular to \cite{JS2}, where Gaussian fluctuations
are analysed.

\bigskip
\paragraph{Acknowledgement}
It is a pleasure to express my gratitude to Yan Fyodorov, for emphasising the importance of critical points; to Vadim Gorin, for many helpful explanations, for
sending me a copy of Kerov's habilitation thesis, and for sharing
the preprint \cite{GZ} before publication;
 to Grigori Olshanski, for discussions of the central limit theorems for
Young diagrams and of many other interesting topics,  and for patiently answering
some questions which should be categorised as philosophical
only if philosophy is understood in the colloquial meaning listed third in the
dictionary of Ushakov; and  to Peter Yuditskii, for correspondence  on
the Markov moment problem and its continual analogues. It is equally pleasant to acknowledge my recognition to  Alexey Bufetov, for showing me the 
matrix $\Gamma_n$ used in \ref{transp} and \ref{trace},
to L\'aszl\'o Erd\H{o}s and Dominik Schr\"oder, for  pointing at several mistakes and irresponsible assertions in a preliminary version of this text; and to Christopher Joyner
and Uzy Smilansky, for spotting lapses of various kinds and suggesting numerous improvements, and for explaining in detail some results from the work \cite{JS2}.

\end{document}